\documentclass[a4paper]{amsart}

\usepackage[textsize=tiny]{todonotes}

\usepackage[shortlabels]{enumitem}

\usepackage{amsmath}
\usepackage{amssymb}
\usepackage{amsfonts}
\usepackage{amsthm}
\usepackage{xcolor}
\usepackage{verbatim}
\usepackage{musicography}
\usepackage{newpxtext}
\usepackage{newpxmath}
\usepackage{fullpage}
\usepackage{mathtools}
\usepackage{tikz-cd}
\usepackage{hyperref}
\theoremstyle{plain}
\newtheorem{theorem}{Theorem}[section]
\newtheorem{lemma}[theorem]{Lemma}
\newtheorem{proposition}[theorem]{Proposition}
\newtheorem{corollary}[theorem]{Corollary}

\newtheorem{example}[theorem]{Example}

\newtheorem*{proposition*}{Proposition}
\newtheorem*{lemma*}{Lemma}
\newtheorem*{theorem*}{Theorem}

\theoremstyle{definition}
\newtheorem{definition}[theorem]{Definition}
\newtheorem{construction}[theorem]{Construction}
\newtheorem{remark}[theorem]{Remark}
\newtheorem*{notation}{Notation}

\newcommand{\A}{\ensuremath{\underline{\mathbb{A}}}}
\newcommand{\Z}{\ensuremath{\underline{\mathbb{Z}}}}
\newcommand{\M}{\ensuremath{\underline{M}}}
\newcommand{\R}{\ensuremath{\underline{R}}}
\newcommand{\upi}{\ensuremath{\underline{\pi}}}

\newcommand{\ck}[1]{\ensuremath{C_{2^{#1}}}}
\newcommand{\omegai}[1]{\ensuremath{\omega^{(#1)}}}

\usepackage{xy}
\input xy
\xyoption{all}

\newcommand{\uta}{u_{2\alpha}}
\newcommand{\aal}{a_{\alpha}}
\newcommand{\alk}{a_{\lambda_k}}
\newcommand{\Si}{\Sigma^{-1}}
\newcommand{\ZZ}{\mathbb{Z}}
\newcommand{\RR}{\mathbb{R}}

\DeclareMathOperator{\coker}{coker}
\DeclareMathOperator{\tr}{tr}
\DeclareMathOperator{\res}{res}
\DeclareMathOperator{\id}{id}
\DeclareMathOperator{\ind}{Ind}
\DeclareMathOperator{\im}{Im}
\DeclareMathOperator{\Hom}{Hom}
\DeclareMathOperator{\Ext}{Ext}
\DeclareMathOperator{\Ab}{Ab}

\DeclareRobustCommand{\SkipTocEntry}[5]{}

\title{On the structure of the $RO(G)$-graded homotopy of $H\underline{M}$ for cyclic $p$-groups}
\author{Igor Sikora}
\address{Krakow University of Economics, Rakowicka 27, 31-510 Krakow, Poland}
\email{sikorai@uek.krakow.pl}
\author{Guoqi Yan}
\address{University of Notre Dame, Notre Dame, IN 46556 USA}
\email{gyan@nd.edu}

\begin{document}
\begin{abstract}
    We study the structure of the $RO(G)$-graded homotopy Mackey functors of any Eilenberg-MacLane spectrum $H\M$ for $G$ a cyclic $p$-group. When $\R$ is a Green functor, we define orientation classes $u_V$ for $H\R$ and deduce a generalized gold relation. We deduce the $a_V,u_V$-isomorphism regions of the $RO(G)$-graded homotopy Mackey functors and prove two induction theorems. As applications, we compute the positive cone of $H\A$, as well as the positive and negative cones of $H\Z$. The latter two cones are essential to the slice spectral sequences of $MU^{((C_{2^n}))}$ and its variants.
\end{abstract}

\maketitle
\tableofcontents
\section{Introduction}
Let $G$ be a finite group. In $G$-equivariant topology the role of ordinary cohomology is played by \emph{Bredon cohomology}. These theories are easy to define, but making computations in them is more complicated than in their non-equivariant analogues. This comes mainly from two reasons. Firstly, Bredon theories take coefficients in diagrams of abelian groups, indexed over subgroups of $G$. Secondly, rather than graded over integers, Bredon theories are naturally graded over the ring of representations of $G$ - they are $RO(G)$-graded.

These two reasons are connected. As shown in \cite{MR598689}, a $\mathbb{Z}$-graded Bredon theory extends to $RO(G)$-graded theory if and only if the coefficients are in the form of a \emph{Mackey functor}. As in non-equivariant topology, a Bredon theory with coefficients in a Mackey functor $\M$ is represented by the \emph{Eilenberg-MacLane spectrum} $H\M$. Spectra of this form are ubiquitious in equivariant homotopy theory. 

For instance, in the groundbreaking work of Hill-Hopkins-Ravenel \cite{HHRa} on the Kervaire invariant one problem, the authors studied the $C_8$-equivariant spectrum $MU^{((C_{8}))}=N_{C_2}^{C_8}MU_{\RR}$, where $N_{C_2}^{C_8}$ is usually referred to as the HHR norm functor. With a suitably chosen class $D\in \pi_{19\rho_8}^{C_8}MU_{\RR}$, the spectrum $\Omega=(D^{-1}MU^{((C_8))})^{C_8}$ satisfies the Periodicity, Gap and Detection Theorems, which makes it a powerful tool in detecting the Kervaire invariant one elements in the stable stems. The Gap Theorem is deduced from a Bredon cohomology computation
\[
H^*_G(S^{m\rho_{C_{2^k}}};\Z). 
\]

In the work of Hahn-Shi \cite{Hahn-Shi}, they showed that at the prime $2$, all of the Morava $E$-theories receive real-orientations
\[
MU_{\RR}\to E_n
\]
which lifts to $G$-equivariant maps $N_{C_2}^GMU_{\RR}\to E_n$ where $G$ is any finite subgroup of the Morava stablizer group that contains $C_2$. With a suitably chosen class $D\in \pi^G_{*\rho_G}N_{C_2}^GMU_{\RR}$, this orientation admits a factorization
\begin{equation*}
    \xymatrix{
    N_{C_2}^GMU_{\RR}\ar[r]\ar[d]&E_n\\
    D^{-1}N_{C_2}^GMU_{\RR}\ar@{-->}[ur]
    }
\end{equation*}
which turns the computations in chromatic homotopy theory to computations of the two equivariant spectra on the left. The main tool for the equivariant computations is the slice spectral sequence first invented by Dugger \cite{Duggerthesis} and utilized to a great extent by Hill-Hopkins Ravenel \cite{HHRa}.

By the Slice Theorem of \cite{HHRa}, the slices of $MU^{((C_{2^k}))}$ are of the forms
\[
H\Z\wedge {C_{2^k}}_+\wedge_{H} S^{m\rho_H},\,e\neq H\subset C_{2^k},\, m\geq 0.
\]
The trick in \cite[Cor.10.4]{HHRb} can be used to prove that the slices of $D^{-1}N_{C_2}^GMU_{\RR}$ are again of the above form, except that in this case, we allow $m$ to be negative integers. Thus the inputs of the $RO(G)$-graded slice spectral sequences of the the periodic $D^{-1}MU^{((C_{2^k}))}$ as well as the periodic versions of quotients of $MU^{((C_{2^k}))}$ are given by the positive and negative cones of the $RO(G)$-graded homotopy groups of $H\Z$, i.e, the areas corresponding to homology and cohomology of actual representation spheres. Since the multiplicative structure of the slice spectral sequence is indispensable in determining the differentials, and the two kinds of most important classes, the Euler and orientation classes, are not in regular-representation-graded homotopy, knowing the entire positive cone and negative cone structure is essential even for computing the $\ZZ$-graded homotopy groups of $D^{-1}MU^{((C_{2^k}))}$.

The main goal of this paper is determining the structure of $H\M^G_\star$ when $G$ is a cyclic group of a prime power order. As applications, we utilize our structural theorems to determine the positive and negative cone of $H\Z$, as well as the positive cone of $H\A$. By working $p$-locally, it suffices to concentrate on the prime $2$ as we will explain in Section \ref{basics}. So for the rest of the introduction, $G=C_{2^k}$. 

\addtocontents{toc}{\SkipTocEntry}
\subsection*{Cellular chains for representation spheres}
One of the tools we use for the calculations of $RO(G)$-graded abelian group structure are the cellular chain and cochain complexes of representation spheres. We give explicit descriptions of differentials in these complexes for any representation $V$ and coefficients $\M$ in Theorem \ref{keythm}, which can be summarized as follows:
\begin{theorem*}
    Let $G=C_{2^k}$, the cellular chain $C_*^G(S^V;\M)$ and cochain $C_*^G(S^V;\M)$ for the representation sphere $S^V$ are completely determined for any $G$-Mackey functor $\M$ and follow from its structure.  
\end{theorem*}
This theorem is a wide extension of the result in Section 3 of \cite{HHRb}, who computed these differentials for the constant Mackey functor $\Z$. Having this description we can compute two important entries: $H\M_{|V|-V}$ for $V$ orientable and $H\M_{-V}$. The first one is equal to the top homology group of $S^V$ and is computed in Proposition \ref{prop mackey uv}:
\begin{proposition*}
    The Mackey functor structure of $H\M_{|V|-V}$ is given as follows:
    \[
H\M_{|V|-V}(H)=
\begin{cases}
\M(G_V)^H&\textrm{if }G_V\subset H,\\
\M(H)&\textrm{otherwise.}
\end{cases}
\]
The structure maps are induced from $\M$ and from fixed point Mackey functor of the $G$-module $\M(G_V)$.
\end{proposition*}
An analogous computation for non-orientable $V$ is given in Propositions \ref{prop 1-alpha} and \ref{prop mult by u on 1-alpha}. As noted above, we also compute $H\M_{-V}$, which is equal to the bottom homology group of the $S^V$. This is done in Proposition \ref{prop mackey av}.
\addtocontents{toc}{\SkipTocEntry}
\subsection*{\texorpdfstring{$a_\lambda$}{}-periodicity} Let $S^0\to S^V$ be the inclusion and denote the corresponding class in $\pi_{-V}^G(S^0)$ by $a_V$, as well as its Hurewicz image in $H\R_{-V}^G$ for any Green functor $\R$. For $G=C_{2^k}$, denote by $\lambda$ its faithful irreducible representation. Then multiplication by $a_\lambda$ in $H\M^G_\star$ together with the induction method allows us to reduce computations to groups graded over representations with the total degree $-1$, $0$ and $1$. We firstly note that we can easily describe multiplication by this element using Propositions \ref{prop alambda iso} and \ref{prop action of alambda}, which can be summarized as follows:

\begin{proposition*}
The multiplication map $a_\lambda\colon H\M^G_V\to H\M^G_{V-\lambda}$ is an isomorphism unless the total degree of $V$ is either $0$, $1$ or $2$. Moreover, if the total degree of $V$ is equal to zero then $a_\lambda$ is an epimorphism, and if the total degree of $V$ is $2$ then $a_\lambda$ is a monomorphism.
\end{proposition*}

The implication of this proposition is strong: most classes in $H\M_{\star}^G$ are infinitely divisible by $a_{\lambda}$, or support infinite $a_{\lambda}$-towers. The Mackey functor structure of $H\M_V$ with $|V|=0$ determines the Mackey functors $H\M_{V+\lambda}$ and $H\M_{V-\lambda}$ (for a precise statement, see Proposition \ref{prop action of alambda}). This allows us to concentrate on the degrees where $|V|=-1,0,1$. 

Assume we want to compute $H\M^G_V$. Using the proposition above, we see that if $V$ is not of total degree $-1$, $0$ or $1$, we can add/substract $\lambda$ from it (this corresponds to the division/multiplication by $a_{\lambda}$) until either we hit one of these degrees or the multiplicity of $\lambda$ reaches zero. In the latter case, the obtained representation $W$ is such that $G_W\coloneqq \ker(W)\neq e$, and the computation can be reduced to a cyclic $2$-group of smaller order. In the first case, although in general we can only fit the homotopy Mackey functors in total degrees $-1$ and $1$ into long exact sequences (Proposition \ref{prop deg pm1}), we have complete control of them (as well as the homotopy Mackey functors in total degree $0$) in the positive and negative cone, which enables us to determined the positive cones of $H\Z$ and $H\A$ and the negative cone of $H\Z$. In general, these ideas lead to the two main induction theorems in this paper, see Theorem \ref{thm induction 1} and Theorem \ref{thm induction 2}.

\addtocontents{toc}{\SkipTocEntry}
\subsection*{Induction method}
A Mackey functor $\M$ over an abelian group $G$ consists of an abelian group $\M(H)$ with an action of $G/H$ for every subgroup $H$ of $G$. These abelian groups are connected by \emph{restrictions} and \emph{transfers}. If $G=C_{2^k}$, a $G$-Mackey functor has simple structure, that can be presented by a Lewis diagram:
\[
\begin{tikzcd}
\M(G)\dar[bend right]\\
\M(G')\uar[bend right]\dar[bend right]\\
\phantom{M}\ldots\phantom{M}\uar[bend right]\dar[bend right]\\
\M(e)\uar[bend right]
\end{tikzcd}
\]
Downward looking arrows are the restrictions and upward looking arrows are transfers. Here $G'$ is the unique subgroup of index two.

We observe that if $V$ is a $G$-representation, the $G$-equivariant computations of $H\M^G_V$ can be reduced to the $G/G_V$-equivariant computations of $V$-th homotopy group of an Eilenberg-MacLane spectrum with another coefficient $\M^{\sharp G_V}$. Here $G_V\coloneqq\ker(V)$, and the new Mackey functor is called the \emph{$G_V$-truncation of $\M$} and is obtained from $\M$ by forgetting $\M(K)$, where $K$ are proper subgroups of $G_V$. Therefore if $G_V\neq e$ we reduced the computation to a smaller group. Since in this paper we are concerned with the homotopy of \emph{all} Eilenberg-MacLane spectra, $\M^{\sharp G_V}$ can belong to a different class of Mackey functors from that of $\M$. For example, the Burnside Mackey functor $\A_{C_{2^k}}$ is not constant, and we have
\[
(\A_{C_{2^k}})^{\sharp C_2}\cong \A_{C_{2^{k-1}}}\oplus \Z^*_{C_{2^{k-1}}}.
\]

\addtocontents{toc}{\SkipTocEntry}
\subsection*{Multiplicative structure} If the Mackey functor $\M$ is a \emph{Green functor}, ie, a commutative ring in the category of Mackey functors, then the coefficients $H\M^G_\star$ is a $RO(G)$-graded ring. The methods presented in the paper allow for the analysis of this ring. The key ingredients are Euler classes $a_V$, orientation classes $u_V$ and relations between these elements. In Proposition \ref{prop av commute} we prove that for any $G$-representations $V,W$ the classes $a_V$ and $a_W$ commute, as well as classes $u_V$ and $u_W$ by Proposition \ref{prop uv commute}. The relation between the classes $a_V$ and $u_W$ is called the \emph{gold relation} and was first observed by Hill-Hopkins-Ravenel in \cite{HHRb} in the case of the constant Mackey functor $\underline{\mathbb{Z}}$. We generalize this relation to the following in Proposition \ref{prop au relation}:
\begin{proposition*}
Let $V,W$ be chosen from $\lambda_i$ for $1\leq i\leq k$ such that $G_V\subset G_W$. Let $R$ be a Green functor. Then there is the following relation in $H\R_{\star}^G$
\[
a_V\cdot\tr_{G_V}^{G_W}(1)u_W=a_W\cdot u_V.
\]
\end{proposition*}

\addtocontents{toc}{\SkipTocEntry}
\subsection*{Positive and negative cone} We conclude the paper with computations of the positive cone and negative cone for any $G$-Mackey functor $\M$. This serves as a presentation how the induction method can be used for calculations. In particular, we provide descriptions of both cones in the case of $H\Z$ in Theorem \ref{thm pos cone HZ} and Theorem \ref{thm neg cone HZ}, as well as the positive cone of $H\A$ in Theorem \ref{thm pos cone HA}.

\addtocontents{toc}{\SkipTocEntry}
\subsection*{Related work} Computations of coefficients of Eilenberg-MacLane spectra over cyclic 2-groups have long history. For the group $C_2$, basing on an unpublished work of Stong, Lewis computed $H\A^{C_2}_\star$, where $\underline{\mathbb{A}}$ is the Burnside Mackey functor, in \cite[Section 2]{MR979507}. Calculations for the constant Mackey functor $\underline{\mathbb{F}}_2$ by Caruso can be found in \cite{MR1684248} and by Hu-Kriz in \cite[Proposition 6.2]{MR1808224}. The computations for the constant Mackey functor $\underline{\mathbb{Z}}$ may be found in Dugger's work \cite[Appendix B]{Duggerthesis}. The full $RO(C_2)$-graded Mackey functor valued coefficients of $H\M$ for any Mackey functor $\M$ was given by Ferland in \cite{MR2699528} and may also be found in Ferland-Lewis \cite[Chapter 8]{MR2025457}. Using the method based on the Tate square as developed by Greenlees-May in \cite{GM95}, Greenlees computed the coefficients of $H\underline{\ZZ}$ in \cite{FourApproaches}.  Basing on this work, the first author computed coefficients of $H\M$ for any $C_2$-Mackey functor $\M$ with an insight into multiplicative structure in \cite{Sikora22}.

The computations for cyclic $2$-groups of higher order were focused so far on the cases of $\underline{\mathbb{Z}}$ and $\underline{\mathbb{F}}_2$. For the first case, partial description is provided by Hill-Hopkins-Ravenel in \cite{HHRb}. The full structure of $RO(G)$-graded ring $H\underline{\mathbb{Z}}^{C_4}_\star$ was computed by Nick Georgakopoulos in \cite{NickG}. The second author computed the same ring using the Tate square method in \cite{Yan22HZ}. In \cite{Yan23HF2}, he also computed the ring structure of coefficients of $H\underline{\mathbb{F}}_2$ over any cyclic $2$-group. Using Tate square methods, the coefficients of $H\underline{\ZZ}$ over groups $C_{p^2}$ with $p$ odd are computed by Zeng in \cite{Zeng17}. The work of Ayala, Mazel-Gee and Rozenblyum \cite{DerivedCpn} provides an $\infty$-categorical calculation in the odd primary case.

\subsection{Notation and conventions}
\label{subsec notations}
Unless indicated otherwise, $G=C_{2^k}$ for $k\geq 1$, a cyclic group of order a power of $2$. By $G'$ we will denote the subgroup of $G$ of the index $2$. A general Mackey functor over $G$ will be denoted by $\M$, and a general Green functor will be denoted by $\R$.

For a Mackey functor $\M$, we denote by $H\M_{\star}:=\pi_{\star}(H\M)$, the Mackey functor valued $RO(G)$-graded homotopy groups of the $G$-spectrum $H\M$. By $H\M^H_V$ we mean the value of this  Mackey functor at $G/H$-level. Structural maps of the Mackey functor $H\M_V$ for a given representation $V$ will be denoted by $\res(V)^H_K$ and $\tr(V)^H_K$. When calculating $G$-homotopy groups of $H\M$, the structural maps of $\M$ will be denoted by $\res^H_K$ and $\tr^H_K$ for $K\subset H$. 

Let $\R$ be a general Green functor. Then its associated Eilenberg-MacLane spectrum is a ring spectrum. In particular, there is a unit map $\eta\colon S^0\to H\R$ and multiplication map $\mu:H\R\wedge H\R\to H\R$.

If $x\in\mathbb{Z}[G]$ and $M$ is a $\mathbb{Z}[G]$-module, by $\prescript{}{x}M$ we denote the submodule of $x$-torsion elements, ie,
\[
\prescript{}{x}M:=\{m\in M\;|\;xm=0\}.
\]

In the more concrete computations of the positive and negative cones, we adopt the following notational conventions as in \cite{Zeng17}, \cite{Yan22HZ} and \cite{Yan23HF2}: $[c]$ means a polynomial generator while $\langle a,b\rangle$ means additive generators. For example, $\ZZ/2[c]\langle a,b\rangle$ means we have elements of the form $c^ia,c^ib, i\geq 0$ all of whom are 2-torsion. By $\ZZ/2\langle a^i\rangle_{i\geq 1}\langle b^j\rangle_{j\geq 1}$ we mean that there are elements of the form $a^ib^j,i,j\geq 1$ all of whom are 2-torsion. The notation $x^{-1}$ means divisibility, ie, $x^{-1}y$ (or equally $\frac{y}{x}$) means an element that satisfies $x\cdot{x^{-1}y}=y$. The notation $[\frac{1}{x}]\langle z\rangle$ means we have the elements $z,\frac{z}{x},\frac{z}{x^2},\frac{z}{x^3},\cdots$. Note some elements of the form $2^ia$. In some cases they are just formal symbols, not 2 times of an actual element $a$. The number $2^i$ in $2^ia$ means that its image under $\res^G_e$ (equally its image under the Borel completion map $H\M\to F(EG_+,H\M)$) is $2^i$ times an actual class $a$.

\subsection{General assumption}
\label{General assumption}
Our methods work for general Mackey functors $\M$ and Green functors $\R$ over the group $C_{p^k}$. For a simpler $RO(G)$-grading, in the whole paper we will work $p$-locally. See Section \ref{sec rog grading} for more details on the grading. 

We note here that if $\M$ (or $\R$) has the property that $\M(G/H)$ is a finitely generated abelian group for each subgroup $H\subset G$, it can be argued using equivariant K\"{u}nneth and universal coefficient spectral sequences \cite{LM06} that $H\M^H_V$ is also finitely generated for each $V\in RO(G)$. Therefore integral information can be completely recovered from local information. This includes the most interesting cases: the Burnside Green functor $\A$, the constant Mackey functors $\underline{A}$ for $A$ a finitely generated abelian group, $\underline{RO(G)}$ and $\underline{R(G)}$, the real and complex representation Green functors respectively. Note that in the $\underline{A}$ case, $p$-localization is not needed for a simpler grading by a theorem of Hu-Kriz, see \cite[Proposition 4.25]{Zeng17}.

\subsection*{Acknowledgement} The first named author would like to thank Krakow University of Economics, who supported this work under POTENCJAŁ program, financed by a subsidy granted to the University. Both authors want to extend their thanks to the anonymous referee, whose deep insights and thorough review highly increased the readability and quality of the paper.

\section{Basics\label{basics}}
In this section we first introduce the basic notions which will be used in our computations. Then we completely determine the cellular chain and cochains of a representation sphere $S^V$. Using this result, we define Euler and orienations classes and discuss commutativity properties among them.
\subsection{The \texorpdfstring{$RO(G)$}{}-grading}
\label{sec rog grading}
Irreducible real representations of the group $\ck{k}$ are as follows:
\begin{enumerate}
\item The 1-dimensional trivial representation. We will denote this representation by 1.
\item The 1-dimensional sign representation, denoted by $\alpha$.
\item 2-dimensional representations $\lambda(m)$ for $1\leq m \leq 2^k-1$ and $m\neq 2^{k-1}$. The action is given by rotation by an angle $\frac{2m\pi}{2^k}$.
\end{enumerate}
Irreducible real representations of the group $C_{p^k}$ for odd $p$ are as follows:
\begin{enumerate}
\item The 1-dimensional trivial representation. We will denote this representation by 1.
\item 2-dimensional representations $\lambda(m)$ for $1\leq m \leq p^k-1$. The action is given by rotation by an angle $\frac{2m\pi}{p^k}$.
\end{enumerate}
For all primes, put $\lambda_j:=\lambda(p^{k-j})$ for $1\leq j\leq k$. Note that $\lambda_1=2\alpha$ when $p=2$. Our convention in this paper on enumerating the generating representations is chosen to support the induction on the group order. Therefore it is different from the one used in Hill-Hopkins-Ravenel's work \cite{HHRa}, \cite{HHRb} and \cite{HHRc}, Zeng's work \cite{Zeng17} as well as the second author's work in \cite{Yan22HZ}, \cite{Yan23HF2} in that the ordering of the representations is reversed.

In this paper we will focus at $p=2$, because it essentially contains the corresponding odd primary results. The reason is as follows. Since we work $p$-locally, we have the stable equivalence of representation spheres $S^{\lambda(rp^i)}\simeq S^{\lambda(p^i)}$ for $(r,p)=1$, see \cite[p.26]{Zeng17}. This enables us to focus on $1,\lambda_j,1\leq j\leq k$ (and $\alpha$ for $p=2$) without losing any information. For odd $p$, the spheres involved in the computations will be $S^1$ and $S^{\lambda_j},1\leq j\leq k$. For $p=2$, the spheres involved will be $S^1,S^{\alpha}$ and $S^{\lambda_j},2\leq j\leq k$, with $S^{2\alpha}=S^{\lambda_1}$ by convention. As a consequence, any result for $p=2$ will contain the corresponding result for an odd $p$ as a special case, where we only focus on those representations with even numbers of $\alpha$ in it.

The homotopy groups $\pi_{\star}H\underline{M}$ will be graded over the following free abelian group on representations:
\[
RO(C_{2^k}):=\mathbb{Z}\{1,\alpha,\lambda_2,\cdots,\lambda_k\}
\]
which will be referred to as the \emph{$RO(G)$-grading}. More precisely, we will study the structure of the direct sum of the following abelian groups for $V$ running over elements in $RO(G)$
\begin{equation*}
H\M_V=[S^V,H\M]^G,\; V\in
\{a_0+a_1\alpha+a_2\lambda_2+\ldots+a_k\lambda_k\;|\;x,y,a_i\in\mathbb{Z}\textrm{ for all }2\leq i\leq k\}.
\end{equation*}
Here $[-,-]^G$ denotes the homotopy class of maps in genuine $G$-spectra, $\M$ is a Mackey functor over $G$ and $H\M$ is the associated Eilenberg-MacLane spectrum.

\begin{definition}
\label{def representations}
\leavevmode
\begin{enumerate}
\item  An element of $RO(G)$ is called a \emph{virtual representation}. If $V=a_0+a_1\alpha+\sum_{i=2}^k a_i\lambda_i$ such that $a_i\geq 0$ for all $0\leq i\leq k$, we will refer to it as an \emph{(actual) $G$-representation}.  
\item The \emph{total dimension/degree (underlying dimension/degree)} of $V=a_0+a_1\alpha+\sum_{i=2}^k a_i\lambda_i$ is defined to be the sum $|V|=a_0+a_1+2\sum_{i=2}^ka_i$.
\item The \emph{fixed dimension} of $V$ as above is the number $a_0$.
\item For $V\in RO(G)$ with $V^G=0$, let ${\lambda_{min}}(V)$ and ${\lambda_{max}}(V)$ be the first and last irreducible $G$-representations that appear in $V$ with non-zero coefficients in the list $\alpha,\lambda_2,\lambda_3,\cdots,\lambda_k$. When $V$ is clear, we simply use ${\lambda_{min}}$ and ${\lambda_{max}}$. If $V$ is an actual $G$-representation, let $G_V$ be its kernel $\{g\in G\;|\;gv=v\textrm{ for all }v\in V\}$. Note in fact we have $G_V=G_{\lambda_{max}}$. For a general $V\in RO(G)$, define $G_V=G_{\lambda_{max}}$.
\item An actual $G$-representation $V$ is \emph{orientable} if the action of $G$ on $V$ factors through $SO(|V|)$. A virtual $G$-representation $V$ is orientable if it can be written as $V=V'-V''$ such that both $V'$ and $V''$ are orientable representations.
\end{enumerate}
\end{definition}

\begin{definition}
\label{def zeta elements}
For $1\leq j\leq k$ we define the following elements of $\mathbb{Z}[G]$: 
\[
\zeta_j:=\sum_{0\leq i < 2^j}\gamma^i\,\text{and}\,\,\overline{\zeta}_j:=\sum_{0\leq i < 2^j}(-\gamma)^i,
\]
where $\gamma$ is the generator of $G$. We also put $\zeta_0=1$.
\end{definition}
\begin{remark}\label{rem zeta elements}
(1) Note that the element $\zeta_j$ gives the norm element in $\mathbb{Z}[G/C_{2^{k-j}}]$ via the quotient map $\mathbb{Z}[G]\to \mathbb{Z}[G/C_{2^{k-j}}]$. In particular, $\zeta_k$ is the norm element in $\mathbb{Z}[G]$ and will also be denoted by $N$. These elements satisfy the following equations:
\begin{align*}
\zeta_j(1-\gamma)&=1-\gamma^{2^j}\\
\overline{\zeta}_j(1+\gamma)&=1-\gamma^{2^j}.
\end{align*}
(2) For an odd prime $p$, we define the elements $\zeta_j$ in the same way, with $2$ replaced by $p$. In this case, all representations are orientable, thus we do not need an analogue of $\overline{\zeta}_j$ for the differentials in Theorem \ref{keythm}.
\end{remark}

\subsection{Cellular chains and cochains of representation spheres}
In this subsection we will describe the cellular structure of the representation spheres. We begin with recalling the \emph{shearing isomorphism}, which will be used extensively throughout the paper.

\begin{proposition}[Shearing isomorphism, {\cite{Alaska}}, {\cite[Section 4]{SchwedeLecture}}]
\label{shearing}
    Let $G$ be a group and $H$ its subgroup. Then for an $H$-space $B$ and $G$-space $A$ there is a $G$-homeomorphism:
    \[
    (G_+\wedge_H B)\wedge A\cong G_+\wedge_H(B\wedge \res^G_H A),\quad (g,b,a)\mapsto (g,b,g^{-1}a).
    \]
In particular, if $B=S^0$ and $A=S^V$ for some $G$-representation $V$, we have that
$
G/H_+\wedge S^V\cong G_+\wedge_H S^{\res^G_H V}. 
$
\end{proposition}

We now move to the description of the cellular structure of representation spheres. As observed in \cite{HHRc}, the representation sphere $S^V$ for an actual representation $V=y\alpha+\sum_{i=2}^k a_i\lambda_i$ has a simple cellular structure, since the stabilizers has a linear order. Consider $S^{\lambda_i+\lambda_j}$ with $i<j$. Recall that $\lambda_1=2\alpha$. The sphere $S^{\lambda_j}$ has a cellular structure
\begin{equation}
    \xymatrix{
    S^0\ar[rr] &&S^{\lambda_j}_{(1)}\ar[rr]\ar[ld]&&S^{\lambda_j}\ar[ld]\\
    &{G/G_{\lambda_j}}_+\wedge S^1\ar@{-->}[ul] &&{G/G_{\lambda_j}}_+\wedge S^2\ar@{-->}[ul]
    }\label{1}
\end{equation}
where $X_{(m)}$ denotes the $m$-skeleton of a $G$-CW complex. Since $i<j$, we have that $G_{\lambda_j}\subset G_{\lambda_i}$, and 
\[
{G/G_{\lambda_j}}_+\wedge S^{\lambda_i}\cong{G/G_{\lambda_j}}_+\wedge S^2
\]
by the shearing isomorphism (Proposition \ref{shearing}). Smashing $S^{\lambda_i}$ with (\ref{1}), and concatenating with the cellular structure on $S^{\lambda_i}$, we get the cellular filtration of $S^{\lambda_i+\lambda_j}$ as
\begin{equation}
    \label{2}
\begin{tikzcd}[column sep = tiny]
    S^0\ar[rr] &&S^{\lambda_i}_{(1)}\ar[rr]\ar[ld]&&S^{\lambda_i}\ar[ld]\ar[rr]&&S^{\lambda_i}\wedge S^{\lambda_j}_{(1)}\ar[rr]\ar[ld]&&S^{\lambda_i+\lambda_j}\ar[ld]\\
    &{G/G_{\lambda_i}}_+\wedge S^1\ar[ul,dashed] &&{G/G_{\lambda_i}}_+\wedge S^2\ar[ul,dashed] &&{G/G_{\lambda_j}}_+\wedge S^3\ar[ul,dashed] &&{G/G_{\lambda_j}}_+\wedge S^4\ar[ul,dashed]
\end{tikzcd}
\end{equation}

This easily generalizes to a general $G$-representation $V=a_1\alpha+\sum_{i=2}^k a_i\lambda_i$. We record it as a lemma:

\begin{lemma}
For an actual $G$-representation $V=a_1\alpha+\sum_{i=2}^k a_i\lambda_i$, the sphere $S^V$ has a cellular structure that starts with $S^0,{G/G_{\lambda_{min}}}_+\wedge S^1$ and ends with ${G/G_{\lambda_{max}}}_+\wedge S^{|V|}$. It has one equivariant cell in each dimension and the stabilizers are in descending order.\label{CellularLemma}
\end{lemma}

We will now state the theorem which will be fundamental for many computations in this paper. Let $V$ be an actual $G$-representation with $V^G=0$. Let $S^V$ be equipped with the cellular structure as in the above lemma. For a better presentation of the result, write $V=V_0+V_1+\cdots+V_n$ with $V_0\in \{0,\alpha\}$, $V_i\in \{\lambda_1,\cdots,\lambda_k\}$ for $1\leq i\leq n$ and that the kernels of the $V_i$'s are in descending order. Note that $V_0=0$ if $V$ is orientable and $V_0=\alpha$ if $V$ is non-orientable. Since $S^V$ has a canonical base point, the point at infinity, all cellular chains and cochains are reduced in this paper.

\begin{theorem}\label{keythm}
Let $G=C_{2^k}$ and let $V$ be as above. Let $s_i=|V_0+\cdots+V_i|,i\geq 0$. Then the (reduced) cellular chain complex $C_*^G(S^V;\M)$ is as follows:
    \begin{enumerate}
        \item If $V$ is orientable, then $C_0^G=\M(G),C^G_{s_i}=C^G_{s_i-1}=\M(G_{V_i})$ for $i\geq 1$. For the differentials, we have for $i\geq 1$,
        \begin{equation*}
            \begin{aligned}
                &\partial_{s_i}=1-\gamma,\\
                &\partial_{s_i-1}=\zeta_{\log_2|G/G_{V_{i-1}}|}\tr^{G_{V_{i-1}}}_{G_{V_i}}.
            \end{aligned}
        \end{equation*}
        \item If $V$ is non-orientable, then $C_0^G=\M(G),C_1^G=\M(G')$ and $C_{s_i}^G=C_{s_i-1}^G=\M(G_{V_i})$ for $i\geq 1$. For the differentials, we have for $i\geq 1$,
        \begin{equation*}
            \begin{aligned}
                &\partial_{s_i}=1+\gamma,\\
                &\partial_{s_i-1}=\overline{\zeta}_{\log_2|G/G_{V_{i-1}}|}\tr^{G_{V_{i-1}}}_{G_{V_i}}.
            \end{aligned}
        \end{equation*}
    \end{enumerate}
    The (reduced) cellular cochain complex $C^*_G(S^V;\M)$ is as follows:
    \begin{enumerate}
        \item If $V$ is orientable, then $C^0_G=\M(G),C_G^{s_i}=C_G^{s_i-1}=\M(G_{V_i})$ for $i\geq 1$. For the differentials, we have for $i\geq 0$,
    \begin{equation*}
            \begin{aligned}
            &\partial^{s_i+1}=1-\gamma,\\
            &\partial^{s_i}=\zeta_{\log_2|G/G_{V_i}|}\res^{G_{V_i}}_{G_{V_{i+1}}}.
            \end{aligned}
        \end{equation*}
        \item If $V$ is non-orientable, then $C^0_G=\M(G),C^1_G=\M(G')$ and $C^{s_i}_G=C^{s_i-1}_G=\M(G_{V_i})$ for $i\geq 1$. For the differentials, we have $\partial^0=\res^G_{G'}$, and for $i\geq 0$,
        \begin{equation*}
            \begin{aligned}
            &\partial^{s_i+1}=1+\gamma,\\
            &\partial^{s_i}=\overline{\zeta}_{\log_2|G/G_{V_i}|}\res^{G_{V_i}}_{G_{V_{i+1}}}.\\
            \end{aligned}
        \end{equation*}
    \end{enumerate}
    If $G=C_{p^k}$ for $p$ odd, then all representations are orientable. The results on the cellular chain and cochain still hold, with $\zeta_{\log_2|G/G_{V_i}|}$ replaced by its odd primary analogue $\zeta_{\log_p|G/G_{V_i}|}$, see also Remark \ref{rem zeta elements}.
\end{theorem}
\begin{proof}
We will only prove the part for chains. The proof for cochains is analogous. We are going to proceed by induction on the group order.

The theorem is true for $G=C_2$, where the cellular chain $C_*^{C_2}(S^{n\alpha};\M)$ is given by
    \[
    \M(G)\xleftarrow{\tr^G_e} \M(e)\xleftarrow{1-\gamma}\M(e)\xleftarrow{1+\gamma}\cdots\leftarrow\M(e),
    \]
    with the differentials interchanging between $1-\gamma$ and $1+\gamma$ except the beginning one. 

Now fix $k\geq 2$ and assume that the theorem is true for groups $C_{2^j}$ for $j<k$. Let $G=C_{2^k}$ and recall that $G'$ is the subgroup of $G$ such that $[G:G']=2$. Let $W=V_0+\cdots+V_t$ be the biggest subrepresentation of $V$ that does not contain $\lambda_k$. Note that in degrees $\leq |W|$ the complex $C_\ast^G(S^V,\M)$ is the same as $C_\ast^G(S^W;\M)$, and the latter is equal to $C_\ast^{G'}(S^W,\M^{\sharp C_2})$ (for the definition of $\M^{\sharp C_2}$, see Definition \ref{def sharp Mackey functor}). Thus by induction we need to prove the claim for the tail of $C_\ast^G(S^V,\M)$, consisting of the following terms:
\[
    \M(G_W)\leftarrow \M(e)\leftarrow \M(e)\leftarrow\cdots\leftarrow\M(e).
\]

The idea is to make use of the Mackey functor structure on cellular chains. More precisely, let $\underline{S^V}$ denote the cellular structure given in Lemma \ref{CellularLemma}, with one cell in each dimension. Then $i^*_{G'}\underline{S^V}$, the restricted cellular structure, has one $0$-cell ${G'/G'}_+$, and two cells in each higher dimension. That is, if $G/H_+\wedge S^i$ is an $i$-cell of $\underline{S^V}$ for $i\geq 1$, then the corresponding restricted cell to $i^*_{G'}\underline{S^V}$ is 
\[
i^*_{G'}(G/H_+\wedge S^i)=G'/H_+\wedge S^i\vee \gamma G'/H_+\wedge S^i.
\] 
For $H\subset G'$ there is the folding map
    \[
    G_+\wedge_{G'}i^*_{G'}{G/{H}}_+\cong {G/{G'}}_+\wedge {G/{H}}_+\to {G/{H}}_+.
    \]
Applying the contravariant part of $\M$ to this map and using additivity of a Mackey functor, we get an embedding which corresponds to $\res^G_{G'}$ in the cellular chain
    \[
    \M(H)\xrightarrow{\Delta}\mathbb{Z}[G]\otimes_{\mathbb{Z}[G']}i^*_{G'}\M(H)\cong \mathbb{Z}[G/G']\otimes \M(H).
    \]
It sends $x\in\M(H)$ to $1\otimes x+\gamma\otimes x$ in the middle term and then to $1\otimes x+\gamma\otimes \gamma x$ on the right. These embeddings collectively form a chain map because they come from the natural transformation $G/G'_+\wedge -\Rightarrow G/G_+\wedge -$ applied to the cellular filtration of $S^V$. Because of this embedding, we can regard $C_*^G(\underline{S^V};\M)$ as a sub-chain complex of $C_*^{G'}(i^*_{G'}\underline{S^V};\M)$ and are left to show that the later has the stated differentials.

Now $i^*_{G'}S^V=S^{i^*_{G'}V}$ as a representation sphere for $G'$ also has the cellular structure described in Lemma \ref{CellularLemma}. We use $\underline{i^*_{G'}S^V}$ to refer to this cellular structure. Both complexes $C_*^{G'}(\underline{i^*_{G'}S^V};\M)$ and $C_*^{G'}(i^*_{G'}\underline{S^V};\M)$ compute the homology of $S^{i^*_{G'}V}$ with coefficient in $\M\downarrow^G_{G'}$ (the restricted Mackey functor to $G'$, see Definition \ref{restricted Mackey functor}). By the inductive assumption we know how to compute the homology of the former, so this forces the differentials in the latter chain complex. Furthermore, we will construct a quasi-isomorphism 
\[
\psi\colon C_*^{G'}(\underline{i^*_{G'}S^V};\M)\xrightarrow{\simeq} C_*^{G'}(i^*_{G'}\underline{S^V};\M)
\] 
in Lemma \ref{lemma of psi}. The differentials in $C_*^{G'}(i^*_{G'}\underline{S^V};\M)$ are also determined there. Then from the embedding
    \begin{equation}
    C_n^{G}(\underline{S^V};\M)\xhookrightarrow{\Delta} C_n^{G'}(i^*_{G'}\underline{S^V};\M)\cong \mathbb{Z}[G/G']\otimes C_n^{G}(\underline{S^V};\M),\, n\geq 1, \label{embedding}
    \end{equation}
    we deduce the differentials in $C_*^{G}(\underline{S^V};\M)$.
\end{proof}

\begin{lemma}\label{lemma of psi}
There exists a quasi-isomorphism
\[
\psi\colon C_*^{G'}(\underline{i^*_{G'}S^V};\M)\xrightarrow{\simeq} C_*^{G'}(i^*_{G'}\underline{S^V};\M).
\]    
\end{lemma}

\begin{proof}
We proceed by the induction on the group order. For the group $C_2$, the subgroup $G'$ is the trivial group. Let $V=m\alpha,m\geq 1$. For $i^*_{e}S^V=S^{|m|}$, recall that we use the reduced cellular structure such that there is only one cell in dimension $m$. We have the following diagram
{\scriptsize
    \[
    \xymatrix{
    &0\ar@{-->}[d]&0\ar@{-->}[d]\ar[l]&0\ar@{-->}[d]\ar[l]&0\ar@{-->}[d]\ar[l]&\cdots\ar[l]&0\ar[l]\ar@{-->}[d]&\M(e)\ar[l]\ar@{-->}[d]^{\psi_m}\\
    &\M(e)&\ind^{C_2}_{e}\M(e)\ar[l]_-{\nabla}&\ind^{C_2}_{e}\M(e)\ar[l]_-{1-\gamma}&\ind^{C_2}_{e}\M(e)\ar[l]_-{1+\gamma}&\cdots\ar[l]&\ind^{C_2}_{e}\M(e)\ar[l]&\ind^{C_2}_{e}\M(e)\ar[l],
    }
    \]
    }
where $\psi_m$ is the $m$-th component of $\psi$. It is defined by
\begin{equation*}
    \psi_m(z)=
    \begin{cases}
        (1,\gamma )(z)=1\otimes z+\gamma\otimes \gamma z & \text{for $m$ even,}\\
        (1,-\gamma )(z)=1\otimes z-\gamma\otimes \gamma z & \text{for $m$ odd}.
    \end{cases}
\end{equation*}
The differentials in the second row are given by
\begin{equation*}
    \begin{cases}
        \nabla(1\otimes z)=z, \nabla(\gamma\otimes z)=z, \\
        (1-\gamma)(1\otimes z)=1\otimes z-\gamma\otimes \gamma z,(1-\gamma)(\gamma\otimes z)=\gamma\otimes z-1\otimes \gamma z,\\
        (1+\gamma)(1\otimes z)=1\otimes z+\gamma\otimes \gamma z,(1+\gamma)(\gamma\otimes z)=\gamma\otimes z+1\otimes \gamma z.
    \end{cases}
\end{equation*}
It is easily seen that $\psi$ is a quasi-isomorphism. 

Now let $G=C_{2^k}$ and assume that $\psi$ has been constructed for groups $C_{2^j}$ with $j<k$. Let $W$ be as in Theorem \ref{keythm} so that $G_W\neq e$. Note that $W$ can be regarded as a $G'\cong G/C_2$-representation. By induction, we can assume that the maps
\[
\psi_n:C_n^{G'}(\underline{i^*_{G'}S^W};\M)\to C_n^{G'}(i^*_{G'}\underline{S^W};\M)
\]
are constructed for $n\leq |W|$.

To construct $\psi_i$ for $i>|W|$, we are going to consider three cases. We firstly consider the case when $V_0=0$. If $W\neq 0$, suppose that $G_W=C_{2^{k-i}}$. Recall that $\zeta_j=\sum_{0\leq i < 2^j}\gamma^i$ (see \ref{def zeta elements}) and $\ind_{H}^G M=\mathbb{Z}[G]\otimes_{\mathbb{Z}[H]} M$ for $M$ a $H$-module. We are to determine the differentials in the second row and construct vertical chain maps that is a quasi-isomorphism in the following diagram:
    {\scriptsize
    \[
    \xymatrix{
    \cdots&\M(C_{2^{k-i}})\ar[l]_-{1-\gamma^2}\ar[d]^{(1,\gamma )}&\M(e)\ar@{-->}[d]\ar[l]_-{\zeta_{i-1}(\gamma^2)\tr^{C_{2^{k-i}}}_e}&\M(e)\ar@{-->}[d]\ar[l]_-{1-\gamma^2}&\M(e)\ar@{-->}[d]\ar[l]_-{\zeta_{k-1}(\gamma^2)}&\M(e)\ar@{-->}[d]\ar[l]_-{1-\gamma^2}&\cdots\ar[l]\\
    \cdots&\ind^G_{G'}\M(C_{2^{k-i}})\ar[l]_-{1-\gamma}&\ind^G_{G'}\M(e)\ar[l]&\ind^G_{G'}\M(e)\ar[l]&\ind^G_{G'}\M(e)\ar[l]&\ind^G_{G'}\M(e)\ar[l]&\cdots\ar[l].
    }
    \]
    }
Here $\zeta_j(\gamma^2)=\sum_{0\leq i < 2^j}(\gamma^2)^i$. 

Note that the induction $\ind^G_H$ is with respect to the Weyl $G/H$-action on $\M(H)$, which comes from the $1\times \gamma$ on action on $G/G'\times G/H$. On the other hand, there is a Weyl $G/G'$-action on $C_*^{G'}(i^*_{G'}\underline{S^W};\M)$, which comes from the $\gamma\times 1$ action on $G/G'\times G/H$.  We have for $H\subset G'$ 
    \[G/G'\times G/H=G/H_{(e,e)}\coprod G/H_{(\gamma,e)},\] 
    where $1\times \gamma$ send $(e,e)$ to $(e,\gamma)=\gamma(\gamma,e)$. Thus it acts as permuting the two factors with an internal twist. And $\gamma\times 1$ acts by permuting the two factors without any internal twist. We note the differentials in the second row are $\gamma\times 1$-equivariant.
    
Denote by $(1,\gamma )$ the map sending $x\in\M(C_{2^{k-i}})$ to $1\otimes x+\gamma\otimes \gamma x$, with $\M(C_{2^{k-i}})=C_{d_{k-1}}^{G'}(\underline{i^*_{G'}S^V};\M)$. To make sure that the second row has the same homology as the first row at degree $|W|$, the first unknown differential in the second row must be ${\zeta_{i}\tr^{C_{2^{k-i}}}_e}$. By further comparing higher homology, we know the following differentials must be interchanging between $1-\gamma$ and $\zeta_k$. Since the dashed arrows should be chain maps, the only possibility is that they interchange between $(1,0)$ and $(1,\gamma )$. The completed diagram is:
    {\scriptsize
    \[
    \xymatrix{
    \cdots&\M(C_{2^{k-i}})\ar[l]_-{1-\gamma^2}\ar[d]^{(1,\gamma )}&\M(e)\ar[d]^{(1,0)}\ar[l]_-{\zeta_{i-1}(\gamma^2)\tr^{C_{2^{k-i}}}_e}&\M(e)\ar[d]^{(1,\gamma )}\ar[l]_-{1-\gamma^2}&\M(e)\ar[d]^{(1,0)}\ar[l]_-{\zeta_{k-1}(\gamma^2)}&\M(e)\ar[d]^{(1,\gamma )}\ar[l]_-{1-\gamma^2}&\cdots\ar[l]\\
    \cdots&\ind^G_{G'}\M(C_{2^{k-i}})\ar[l]_-{1-\gamma}&\ind^G_{G'}\M(e)\ar[l]_-{\zeta_{i}\tr^{C_{2^{k-i}}}_e}&\ind^G_{G'}\M(e)\ar[l]_-{1-\gamma}&\ind^G_{G'}\M(e)\ar[l]_-{\zeta_{k}}&\ind^G_{G'}\M(e)\ar[l]_-{1-\gamma}&\cdots\ar[l].
    }
    \]
    }
    Here by $\gamma$ we mean the $1\times \gamma$-action as mentioned above.
    The differentials in the second row are given by
    \begin{align*}
        &({\zeta_{i}\tr^{C_{2^{k-i}}}_e})(1\otimes x)=1\otimes\zeta_{i-1}(\gamma^2)\tr^{C_{2^{k-i}}}_e(x)+\gamma\otimes\gamma\zeta_{i-1}(\gamma^2)\tr^{C_{2^{k-i}}}_e(x),\\
        &({\zeta_{i}\tr^{C_{2^{k-i}}}_e})(\gamma\otimes y)=1\otimes\gamma\zeta_{i-1}(\gamma^2)\tr^{C_{2^{k-i}}}_e(y)+\gamma\otimes\zeta_{i-1}(\gamma^2)\tr^{C_{2^{k-i}}}_e(y),\\
        &(1-\gamma)(1\otimes x)=1\otimes x-\gamma\otimes \gamma x,\\
        &(1-\gamma)(\gamma\otimes y)=\gamma\otimes y-1\otimes \gamma y\\
        &\zeta_k(1\otimes x)=1\otimes\zeta_{k-1}(\gamma^2)x+\gamma\otimes\gamma\zeta_{k-1}(\gamma^2)x,\\
        &\zeta_k(\gamma\otimes y)=1\otimes\gamma\zeta_{k-1}(\gamma^2)x+\gamma\otimes\zeta_{k-1}(\gamma^2)x.
    \end{align*}
    The vertical maps are given by
    \begin{align*}
        &(1,\gamma )(z)=1\otimes z+\gamma\otimes \gamma z,\\
        &(1,0)(z)=1\otimes z.
    \end{align*}
    The embedding $(\ref{embedding})$ is induced by $*\times G/H\xleftarrow{*\times 1}G/G'\times G/H$ for $H\subset G'$. In particular, it sends $x\in \M(e)$ to $1\otimes x+\gamma\otimes x\in\ind^G_{G'}\M(e)$, from where we deduce all the differentials of $C_*^{G}(\underline{S^V};\M)$.

    For $V_0=0$ and $W=0$, that is $S^V=S^{a_k\lambda_k}$, we start from
    \[
    \xymatrix{
    \M(G')\ar[d]^1&\M(e)\ar[l]_-{\tr^{G'}_e}\ar@{-->}[d]^{(1,0)}&\M(e)\ar@{-->}[d]^{(1,\gamma )}\ar[l]_-{1-\gamma^2}&\M(e)\ar@{-->}[d]^{(1,0)}\ar[l]_-{\zeta_{k-1}(\gamma^2)}&\cdots\ar[l]\\
    \M(G')&\ind^G_{G'} \M(e)\ar[l]_-{\tr^{G'}_e(1\oplus\gamma )}&\ind^G_{G'} \M(e)\ar[l]&\ind^G_{G'} \M(e)\ar[l]&\cdots\ar[l]
    }
    \]
    and deduce the differentials and chain maps as in the previous case. Here the map $\tr^{G'}_e(1\oplus\gamma)$ sends $1\otimes x+\gamma\otimes y$ to $\tr^{G'}_e(x+\gamma y)$. It is deduced from the map $G/G'\times *\xleftarrow{1\times *} G/G'\times G/e$.
    
    Now for the $V_0=\alpha$   case. Assume $G_W=C_{2^{k-i}}$, then it is contained in $G'$. Using the same method, we deduce the completed diagram
    {\scriptsize
    \[
    \xymatrix{
    \cdots&\M(C_{2^{k-i}})\ar[l]_-{1+\gamma^2}\ar[d]^{(1,-\gamma )}&\M(e)\ar[d]^{(1,0)}\ar[l]_-{\overline{\zeta}_{i-1}(\gamma^2)\tr^{C_{2^{k-i}}}_e}&\M(e)\ar[d]^{(1,-\gamma)}\ar[l]_-{1+\gamma^2}&\M(e)\ar[d]^{(1,0)}\ar[l]_-{\overline{\zeta}_{k-1}(\gamma^2)}&\M(e)\ar[d]^{(1,-\gamma )}\ar[l]_-{1+\gamma^2}&\cdots\ar[l]\\
    \cdots&\ind^G_{G'}\M(C_{2^{k-i}})\ar[l]_-{1+\gamma}&\ind^G_{G'}\M(e)\ar[l]_-{\overline{\zeta}_i\tr^{C_{2^{k-i}}}_e}&\ind^G_{G'}\M(e)\ar[l]_-{1-\gamma}&\ind^G_{G'}\M(e)\ar[l]_-{\zeta_k}&\ind^G_{G'}\M(e)\ar[l]_-{1-\gamma}&\cdots\ar[l].
    }
    \]
    }
    where $(1,-\gamma )(z)=1\otimes z-\gamma\otimes \gamma z$.

\end{proof}

\begin{remark}\label{remark Mackey functor from cellular chain}
    We have also showed how to deduce the homology Mackey functors using the cellular chains $C_*^{H}(\underline{i^*_H S^V};\M)$ for different $H\subset G$. Note that these chains have one cell in each dimension, therefore are easier to compute. Since restrictions and transfers are transitive, we can take $H=G'$. The maps $\Delta$ and $\nabla$ in the following diagram are usually used to deduce $\res^G_{G'}$ and $\tr^G_{G'}$:
    \[
    \xymatrix{
    C_*^{G}(\underline{S^V};\M)\ar@/^/[rd]^{\Delta}\ar@{-->}[d]|-{\times}\\
    C_*^{G'}(\underline{i^*_{G'}S^V};\M)\ar[r]_{\psi}& C_*^{G'}(i^*_{G'}\underline{S^V};\M)\ar@/^/[ul]^{\nabla}.
    }
    \]
    It is not obvious that there exist restrictions and transfers between the two chains of abelian groups on the left. In fact, there is no chain map in the direction of the dashed arrow that makes the triangle commute. However, since $H_*(\psi)$ is an isomorphism, we can deduce the restrictions and transfers between the homology of the two chains on the left by first computing $H_*(\Delta)$ and $H_*(\nabla)$.
\end{remark}

\begin{remark}
    The motivations for the differentials in Theorem \ref{keythm} are simple: they are the universal formulas for all Mackey functors, thus they should satisfy the least amount of restrictions that make the cellular chain/cochain complexes.
\end{remark}

\begin{example}
    The chain complex $C_*(S^{\alpha+2\lambda_k};\M)$ is given by
    \begin{equation*}
        \begin{aligned}
            &0\leftarrow\M(G)\xleftarrow{\tr^G_{G'}}\M(G')\xleftarrow{(1-\gamma)\tr^{G'}_e}\M(e)\xleftarrow{1+\gamma}\M(e)\xleftarrow{\overline{\zeta}_k}\M(e)\xleftarrow{1+\gamma}\M(e).
        \end{aligned}
    \end{equation*}
\end{example}

\begin{example}
    The cochain complex $C^*(S^{2\alpha+2\lambda_2+2\lambda_3};\M)$ with $k\geq 3 \text{ in } C_{2^k}$ is given by
    \begin{equation*}
        \begin{aligned}
            &0\to \M(G)\xrightarrow{\res^G_{G'}}\M(G')\xrightarrow{1-\gamma}\M(G')\xrightarrow{(1+\gamma)\res^{G'}_{C_{2^{k-2}}}}\M(C_{2^{k-2}})\xrightarrow{1-\gamma}\M(C_{2^{k-2}})\xrightarrow{\zeta_2}\M(C_{2^{k-2}})\xrightarrow{1-\gamma}\M(C_{2^{k-2}})\\
            &\xrightarrow{\zeta_2\res^{C_{2^{k-2}}}_{C_{2^{k-3}}}}\M(C_{2^{k-3}})\xrightarrow{1-\gamma}\M(C_{2^{k-3}})\xrightarrow{\zeta_3}\M(C_{2^{k-3}})\xrightarrow{1-\gamma}\M(C_{2^{k-3}}).
        \end{aligned}
    \end{equation*}
\end{example}

\subsection{Induction method}
In this subsection we will describe the induction method, which enables us to reduce computations to smaller groups when $G_V\neq e$.

\begin{definition}
\label{def sharp Mackey functor}
    Let $\M$ be a $G$-Mackey functor and let $H\leq G$. Then an \emph{$H$-truncation of $\M$} is the $G/H$-Mackey functor $\M^{\sharp H}$ defined by
    \[
\M^{\sharp H}(K/H):= \M(K)
    \]
for $H\leq K\leq G$. The transfers, restrictions, and Weyl group actions are the respective structural maps of $\M$.
\end{definition}
\begin{remark}
    Intuitively, the $H$-truncation of $\M$ amounts to forgetting all of the information below the $H$-level of the Mackey functor $\M$.
\end{remark}
\begin{example}
    Let $\A(C_4)$ be the Burnside Mackey functor over $C_4$. Then $\A(C_4)^{\sharp C_2}$ is the $C_4/C_2$-Mackey functor
    \[
\begin{tikzcd}
    \mathbb{A}(C_4)\dar[bend right,"\res_{\A}"'] \\
    \mathbb{A}(C_2).\uar[bend right,"\tr_{\A}"']
\end{tikzcd}
    \]
\end{example}

The basis for retrieving $G$-information on $H\M^G_\star$ from the $H$ and $G/H$ computations is the following result:
\begin{proposition}
    Let $H\leq G$ and let $(H\M)^H$ be the $H$-categorical fixed points spectrum of $H\M$. Then there is the following equivalence of $G/H$-spectra:
    \[
    (H\M)^H\simeq H(\M^{\sharp H}).
    \]
\end{proposition}
\begin{proof}
    Comes from the following chain of isomorphsims, for $H\leq K\leq G$:
    \[
    \pi^{K/H}_n(H\M^H)=[S^n,H\M^H]^{K/H}\cong [S^n,H\M]^K =\pi^K_n(H\M).
    \]
    For $n=0$, we then obtain that
    \[
    \pi^{K/H}_0(H\M^H)\cong \pi^K_0(H\M)=\M(K)=\M^{\sharp H}(K/H)=\pi^{K/H}_0(H\M^{\sharp H}).
    \]
\end{proof}

Now consider a $G$-representation $V$. Since $G_V$ acts trivially on it, we can regard it as a $G/G_V$-representation. If $G_V\neq e$, then $a_k=0$ and $G/G_V$ is a cyclic $2$-group of smaller order. Then we can reduce to $G/G_V$-computations by the following corollary:
\begin{corollary}
\label{reduction to quotient group}
    Let $V$ be such that $G_V\neq e$. Then
    \[
    H\M^G_{V}\cong (H\M^{\sharp G_V})^{G/G_V}_V.
    \]
    The above isomorphism also preserves the Mackey functor structure. For $G_V\subset K\subset H$, we have 
    \[
    \res_K^H\leftrightarrow\res_{K/{G_V}}^{H/{G_V}}\quad\text{and}\quad \tr_K^H\leftrightarrow\tr_{K/{G_V}}^{H/{G_V}}.
    \]
\end{corollary}

\subsection{Classes \texorpdfstring{$a_V$}{} and \texorpdfstring{$u_V$}{}}
In this section we will define the classes $a_V$ and $u_V$. We start with the definition of Euler classes $a_V$. Recall that $\M$ denotes an arbitrary Mackey functor and $\R$ is an arbitrary Green functor.
\begin{definition}
\label{def a classes}
Let $V$ be an actual irreducible $G$-representation and consider the inclusion of fixed points $i\colon S^0\to S^V$. We define the element $a_V\in H\R^G_{-V}$ to be the class of the map $S^0\xrightarrow{i}S^V\cong S^V\wedge S^0\xrightarrow{1\wedge \eta}S^V\wedge H\R$. It factors through the ring maps $S^0\to H\A\to H\R$. This class is well-known as the Euler class of $V$. Note that $a_V=0$ if $V^G\neq 0$ since $S^0\to S^V$ will be equivariantly null-homotopic.
\end{definition}

\begin{remark}
Let $W\in RO(G)$, $V$ be an irreducible representation and let $x\in H\M_W^G$. Then multiplication by the element $a_V$ is given by
\[
\begin{tikzcd}[column sep = large]
S^W\cong S^0\wedge S^W\rar["a_V\wedge x"]& S^{V}\wedge H\M,
\end{tikzcd}
\]
or
\[
\begin{tikzcd}[column sep = large]
S^W\cong S^0\wedge S^W\rar["a_V\wedge x"]& S^V\wedge H\A \wedge H\M\rar["1\wedge \mu"]& S^V\wedge H\M.
\end{tikzcd}
\]
where the second map uses the $H\A$-module structure on $H\M$.
\end{remark}

We can compute the homotopy Mackey functors at the degrees of the Euler classes:

\begin{proposition}
\label{prop mackey av}
Let $V=a_1\alpha+\sum_{i=2}^k a_i\lambda_i$ be an actual $G$-representation. Recall that $\lambda_{min}$ is the first irreducible representation with non-zero coefficient in the list $\alpha,\lambda_2,\ldots,\lambda_k$ (see Definition \ref{def representations}). Then $H\M^G_{-V}\cong \M(G)/\tr^{G}_{G_{\lambda_{min}}}$. For the Mackey functor structure, we have
\[
H\M_{-V}(H)=\left\{
\begin{array}{ll}
\coker(\tr^H_{G_{\lambda_{min}}})&\textrm{if }G_{\lambda_{min}}\subsetneq H,\\
0&\textrm{otherwise }.
\end{array}
\right.
\]
 For $K\subset H$, if $K\subset G_{\lambda_{min}}$, then $\res^H_K,\tr^H_K$ are both $0$. For $G_{\lambda_{min}}\subset K\subset H$, the maps $\res^H_K$ and $\tr^H_K$ are induced from that of $\M$.
\end{proposition}
\begin{proof}
The first claim directly comes from the cellular chain of $S^V$ from Lemma \ref{CellularLemma}.

For the Mackey functor structure, if $H\subset G_{\lambda_{min}}$, then $H\M_{-V}^{G_{\lambda_{min}}}\cong H\M_{-|V|}^{G_{\lambda_{min}}}=0$ by the shearing isomorphism (see Proposition \ref{shearing}) and the definition of $H\M$.

Now take $G_{\lambda_{min}}\subset K\subset H$, the first two terms of the $G/H$ and $G/K$-level cellular chains are
\[
\xymatrix{
0&\M(G/H\times G/G)\ar[l]\ar@/_/[d]_{\res^H_K}&\M(G/H\times G/{G_{\lambda_{min}}})\ar[l]_{1\times \nabla}\\
0&\M(G/K\times G/G)\ar[l]\ar@/_/[u]_{\tr^H_K}&\M(G/K\times G/{G_{\lambda_{min}}})\ar[l]_{1\times \nabla}
}
\]
where $\nabla:G/{G_{\lambda_{min}}}\to G/G$ the canonical projection of $G$-sets. We have the identifications
\[
\M(G/H\times G/{G_{\lambda_{min}}})\cong \M(G\times_H i^*_H G/{G_{\lambda_{min}}})=\M(G\times_H \underset{G/H}{\coprod}H/{G_{\lambda_{min}}})\cong \underset{G/H}{\oplus}\M(G/{G_{\lambda_{min}}})
\]
and the map $1\times \nabla$ is the sum of $\tr^H_{G_{\lambda_{min}}}$'s. Similarly for the $G/K$-level chain complex. Taking $0$-th homology proves the claim.
\end{proof}

\begin{remark}
    It will be proved later that the map $a_V:H\M_0^G\to H\M_{-V}^G$ is the projection on the cokernel of $\tr^G_{G_V}$ (see Proposition \ref{prop hmnV}). For a Green functor $\R$, we have $a_V=\bar{1}\in \R(G)/{\tr^G_{G_V}}$, the class represented by $1\in \R(G)$.
\end{remark}

\begin{example}
If $G=C_4$, $\R=\A$, we have
\begin{equation*}
\xymatrix{
H\A_{-\alpha}=&\A(C_4)/\tr^{C_4}_{C_2}\ar@/_/[d]_{0}&H\A_{-\lambda}=&\A(C_4)/\tr^{C_4}_e\ar@/_/[d]_{\overline{\res^{C_4}_{C_2}}}\\
&0\ar@/_/[u]_{0}\ar@/_/[d]_{0}&&\A(C_2)/\tr^{C_2}_e\ar@/_/[d]_{0}\ar@/_/[u]_{\overline{\tr^{C_4}_{C_2}}}\\
&0\ar@/_/[u]_{0}&&\ar@/_/[u]_{0}0
}
\end{equation*}
Denote $\A(C_4)=\mathbb{Z}\langle 1,c,d \rangle$ where $c=[C_4/{C_2}],d=[C_4/e]$, and $\mathbb{A}(C_2)=\mathbb{Z}\langle 1,\omega\rangle$, where $\omega=[C_2/e]$. Then $H\A_{-\lambda}(G/G)=\mathbb{Z}\langle a_{\lambda}\rangle\oplus \mathbb{Z}\langle ca_{\lambda}\rangle$, and $H\A_{-\lambda}(G/{C_2})=\mathbb{Z}\langle \res^{C_4}_{C_2}a_{\lambda}\rangle$. We have $\overline{\tr^{C_4}_{C_2}}(\res^{C_4}_{C_2}a_{\lambda})=ca_{\lambda}$.
\end{example}

Now we turn our attention to the second family of generators, the orientation classes $u_V$ which will be defined in Definition \ref{def uv Green functors}. Recall the definition of orientable representations in Definition \ref{def representations}.

\begin{remark}
Note that the only irreducible non-orientable representation of $G$ is the sign representation $\alpha$. In particular, $V=a_0+a_1\alpha+\sum_{2\leq i\leq k}a_i\lambda_i\in RO(G)$ is non-orientable iff $a_1$ is odd and every non-orientable representation $V$ can be written as $V=V'\pm\alpha$ with orientable $V'$. 
\end{remark}

We firstly study the Mackey functor structure on the ($|V|-V$)-degree.
\begin{proposition}
\label{prop mackey uv}
Let $G=C_{2^k}$. Let $V$ be an orientable (not necessarily irreducible) $G$-representation. Then $H\M^G_{|V|-V}\cong \M(G_V)^G$. For the Mackey functor structure, we have
\[
H\M_{|V|-V}(H)=
\begin{cases}
\M(G_V)^H&\textrm{if }G_V\subsetneq H,\\
\M(H)&\textrm{otherwise.}
\end{cases}
\]
For $K\subset H\subset G_V$, the maps $\res^H_K,\tr^H_K$ are induced from that of $\M$. If $G_V\subset K\subset H$, then $\res^H_K$ is induced by inclusion of fixed points, and $\tr^H_K$ is given by $\tr^{C_{2^i}}_{C_{2^{i-1}}}(x)=(1+\gamma^{2^{k-i}})\cdot x$ and transitivity.
\end{proposition}

\begin{proof}
$S^V$ has a cellular filtration with cells in top two dimensions given by $G/{G_V}_+\wedge S^{|V|-1}$ and $G/{G_V}_+\wedge S^{|V|}$. So its chain complex with coefficient in $\M$ has the following form in top dimensions
\begin{equation*}
    \cdots\xleftarrow{}\M(G/G_V)\xleftarrow{1-\gamma}\M(G/G_V).
\end{equation*}
The kernel of the top differential is $\M(G_V)^G$. Thus the first claim follows.

For the claim about the Mackey functor structure, notice that the lattice of subgroups of $G$ has a linear order. By transitivity of restrictions and transfers, we may assume that $[H:K]=2$.

When $H\subset G_V$, the shearing isomorphism $G/H_+\wedge S^{V-|V|}\cong G/H_+\wedge S^0$ (see Proposition \ref{shearing}) implies an isomorphism of Mackey functors $H\M_0\downarrow^G_{G_V}\cong H\M_{|V|-V}\downarrow^G_{G_V}$. Now let $G_V\subset K\subset H=C_{2^i}=\langle\gamma^{2^{k-i}}\rangle$. Thus $K=C_{2^{i-1}}=\langle\gamma^{2^{k-i+1}}\rangle$. Let $\nabla:G/K\to G/H$ be the natural projection. In the following commutative diagram
\[
\xymatrix{
\M(G/H\times G/G_V)\ar[d]_{\nabla\times 1}&\M(G/H\times G/G_V)\ar[l]_{1\times (1-\gamma)}\ar[d]^{\nabla\times 1}\\
\M(G/K\times G/{G_V})&\M(G/K\times G/{G_V}
)\ar[l]_{1\times (1-\gamma)},
}
\]
the first and second rows are the differential $C_{|V|-1}(S^V;\M)\xleftarrow{1-\gamma}C_{|V|}(S^V;\M)$ on the $G/H$ and $G/K$-levels, respectively. We have the identifications
\begin{equation*}
\begin{aligned}
&\M(G/H\times G/G_V)\cong[G/H_+\wedge G/{G_V}_+,H\M]^G\cong [G_+\wedge_H i^*_H(G/{G_V}_+),H\M]^G\cong [i^*_H(G/{G_V}_+),H\M]^H\\
&\cong \mathbb{Z}[G]\otimes_{\mathbb{Z}[H]}i^*_H \M(G_V)    
\end{aligned}
\end{equation*}
and similarly for $\M(G/K\times G/G_V)$. Here the notation $i^*_H \M(G_V)$ means we are restricting the Weyl $G/G_V$-action to an $H/G_V$-action on $\M(G_V)$. To understand the differential $1\times (1-\gamma)$, we have to understand the action of $\gamma$ on $\mathbb{Z}[G]\otimes_{\mathbb{Z}[H]}i^*_H \M(G_V)$ coming from $\gamma:G/G_V\to G/G_V$. Tracing through the identifications we conclude that it acts by left multiplication on $\mathbb{Z}[G]$. 

The commutative diagram above becomes
\[
\xymatrix{
\mathbb{Z}[G]\otimes_{\mathbb{Z}[H]}i^*_H \M(G_V)\ar[d]_{}&\mathbb{Z}[G]\otimes_{\mathbb{Z}[H]}i^*_H \M(G_V)\ar[l]_{(1-\gamma)}\ar[d]^{}\\
\mathbb{Z}[G]\otimes_{\mathbb{Z}[K]}i^*_K \M(G_V)&\mathbb{Z}[G]\otimes_{\mathbb{Z}[K]}i^*_K \M(G_V)\ar[l]_{(1-\gamma)},
}
\]
where both $(1-\gamma)$ are given by left multiplication on $\mathbb{Z}[G]$. The kernel of the two $(1-\gamma)$'s are
\[
(\mathbb{Z}[G]\otimes_{\mathbb{Z}[H]}i^*_H \M(G_V))^G\cong \mathbb{Z}[G/H]^G\otimes \M(G_V)^H\cong \M(G_V)^H,
\]
and
\[
(\mathbb{Z}[G]\otimes_{\mathbb{Z}[K]}i^*_H \M(G_V))^G\cong \mathbb{Z}[G/K]^G\otimes \M(G_V)^K\cong \M(G_V)^K.
\]
The induced maps on kernels is the map induced from the $G$-equivariant map $\mathbb{Z}[G/H]\to \mathbb{Z}[G/K]$ by sending $1$ to $1+\gamma^{2^{k-i}}$, which can be easily seen to be $\mathbb{Z}\xrightarrow{1}\mathbb{Z}$ upon taking $(-)^G$. Thus restrictions are induced by inclusion of fixed points.

For the transfers, we have the double-coset formula
\[
\res^H_K\circ \tr^H_K(x)=(1+\gamma^{2^{k-i}})\cdot x.
\]
Since restrictions are inclusions of subgroups, we deduce $\tr^H_K(x)=(1+\gamma^{2^{k-i}})\cdot x$.

\end{proof}

\begin{construction}
   \label{cons class uv}
For an orientable $G$-representation $V$, consider the following commutative diagram
\[
\xymatrix{
&\M(G)\ar[d]^{\res^G_{G_V}}\ar@{-->}[ld]\\
H\M^G_{|V|-V}\cong \M(G_V)^G\ar[r]&\M(G_V)\ar[r]^{1-\gamma}&\M(G_V)
}
\]
The map $(1-\gamma)$ is the top differential of $C_*(S^V;\M)$. The dashed arrow is induced from $\res^G_{G_V}$, since for any $H\subset G$, we have $\res^G_{H}(\M(G))\subset \M(H)^G$. We make the following definition of orientation classes for Green functors. 
\end{construction}

\begin{definition}
\label{def uv Green functors}
For $V$ an orientable $G$-representation, $\R$ a Green functor, we define the \textit{orientation class} $u_V$ to be the element $\res^G_{G_V}(1)\in H\R^G_{|V|-V}$.
\end{definition}
\begin{remark}
\label{remark action of uv}
\leavevmode
\begin{enumerate}
\item Note that by Construction \ref{cons class uv} we obtain that the multiplication by $u_V\colon H\R^G_0\to H\R^G_{|V|-V}$ is the restriction map $\res^G_{G_V}$, by the following commutative diagram
\[
\xymatrix{
H\R^G_0\ar[r]^{u_V}\ar[d]_{\res^G_{G_V}}&H\R^G_{|V|-V}\ar[d]^{\res(|V|-V)^G_{G_V}}\\
H\R_0^{G_V}\ar[r]^{u_V}&H\R_{|V|-V}^{G_V}
}
\]
where ${\res(|V|-V)^G_{G_V}}$ is inclusion of $G$-fixed points, and the horizontal map in the second row is an isomorphism.
\item Definition \ref{def uv Green functors} is restricted to Green functors. However, since every Mackey functor $\M$ is a module over Burnside Mackey functor $\A$, we can consider the induced action of $H\A^G_\star$ on $H\M^G_\star$. We can still deduce $u_V\colon H\M_0^G\to H\M^G_{|V|-V}$ is $\res^G_{G_V}$ as above. If $\R$ is a Green functor, its orientation class coincides with the image of the orientation class of $H\A$ through the canonical ring map $H\A\to H\R$.
\end{enumerate}
\end{remark}

\begin{example}
\leavevmode
\begin{enumerate}
    \item When $\R=\Z$, $H\Z_{|V|-V}=\Z$, and the definition of orientation classes here agrees with the orientation classes used widely in previous computations of \cite{HHRa,HHRb,Zeng17,NickG,Yan22HZ}.
    \item When $\R=\A$ and $G=C_4$, the Weyl actions are always trivial, and we have
\begin{equation*}
\xymatrix{
H\A_{2-2\alpha}=&\A(C_2)\ar@/_/[d]_{1}\quad&H\A_{2-\lambda}=&\A(e)\ar@/_/[d]_{1}\\
&\A(C_2)\ar@/_/[u]_{2}\ar@/_/[d]_{\res^{C_2}_e}\quad&&\A(e)\ar@/_/[d]_{1}\ar@/_/[u]_{2}\\
&\A(e)\ar@/_/[u]_{\tr^{C_2}_e}\quad&&\ar@/_/[u]_{2}\A(e)
}
\end{equation*}
If $\mathbb{A}(C_2)=\mathbb{Z}\langle 1,\omega\rangle$, where $\omega=[C_2/e]$, then $u_{2\alpha}\in H\A_{2-2\alpha}(G/G)$ correspond the the generator $1$ and $H\A_{2-2\alpha}(G/G)=\mathbb{Z}\langle u_{2\alpha}\rangle\oplus \mathbb{Z}\langle \omega u_{2\alpha}\rangle$. Note this agrees with a general definition of orientation classes for all finite groups of Kriz \cite{SophieKriz}.
\end{enumerate}
\end{example}

The following two propositions allow us to restrict our attention only to classes coming from irreducible representations.
\begin{proposition}
\label{prop av commute}
Let $G=C_{2^k}$ and $V,W$ be two $G$-representations. Then the following equalities hold in $\pi^G_{\star}S^0$, thus hold in $H\R^G_{\star}$:
\begin{equation*}
    a_Va_W=a_Wa_V=a_{V+W}=a_{W+V}.
\end{equation*}
\end{proposition}
\begin{proof}
    $RO(G)$-graded commutativity is studied by Dugger in \cite{Dugger14coherence}. Using his theory, we will show that these equalities actually hold in the $G$-equivariant stable stems. Since for any two $G$-representation $V_1,V_2$ we have $a_{V_1}a_{V_2}=a_{V_1+V_2}$, we can assume that $V,W$ are irreducible. Using the methods in Appendix B of \cite{Yan22HZ}, we can compute for each $\lambda_i$, $\tau(S^{\lambda_i})=1$ and $\tau(S^{\alpha})=1-\tr^{G}_{G'}\circ \res^G_{G'}$. Here $\tau(X)\in \pi^G_0 S^0$ is the trace of identity of a finite $G$-complex $X$ studied in Subsection 4.19 of \cite{Dugger14coherence}.

    Now we will use \cite[Proposition 1.2]{Dugger14coherence}. Take any $x\in \pi^G_{-V}S^0 ,y\in \pi^G_{-W}S^0$. If $V=\lambda_i$ and $W=\lambda_j$ for some $1\leq i,j\leq k$, then $xy=\tau(S^{\lambda_i})^{1\cdot 1}yx=yx$ if $i=j$, and $xy=\tau(S^{\lambda_i})^0\tau(S^{\lambda_j})^0yx=yx$ if $i\neq j$.

    If $V=\lambda_i$ for $1\leq i\leq k$, and $W=\alpha$, then $xy=\tau(S^{\lambda_i})^0\tau(S^{\alpha})^0yx=yx$.

    If $V=W=\alpha$, then $xy=\tau(S^{\alpha})^{1\cdot 1}yx=(1-\tr^{G}_{G'}\circ \res^G_{G'})yx$. The class $\tr^{G}_{G'}\circ \res^G_{G'}yx$ may not be zero in general. But if $x=y=a_{\alpha}$, it becomes zero because $\res^G_{G'}a_{\alpha}=0$.
\end{proof}
\begin{proposition}
\label{prop uv commute}
Let $V,W$ are two orientable $G$-representations. Assume $G_V\subset G_W$. Regard $H\M$ as an $H\A$-module, then:
\begin{enumerate}
\item The map $u_V\colon H\M^G_{|W|-W}\to H\M^G_{|V|+|W|-V-W}$ is the restriction map $\res^{G_W}_{G_V}$.
\item The map $u_W\colon H\M^G_{|V|-V}\to H\M^G_{|V|+|W|-V-W}$ is an isomorphism.
\item We have that $u_Vu_W=u_Wu_V$ in $H\A_{\star}$, thus in $H\R_{\star}$ for all Green functors $\R$.
\end{enumerate}
\end{proposition}
\begin{proof}
Consider the following commutative diagram where horizontal maps are restriction maps in $H\M_{|W|-W}$ and $H\M_{|V|+|W|-V-W}$, and vertical maps are multiplication by $u_V$ and its restrictions:
\[
\xymatrix{
H\M_{|W|-W}^G\ar[r]^{\res^{G}_{G_W}}\ar[d]_{u_V}&H\M_{|W|-W}^{G_W}\ar[r]^{\res^{G_W}_{G_V}}\ar[d]_{u_V}&H\M_{|W|-W}^{G_V}\ar[d]^{\cong}_{u_V}\\
H\M_{|V|+|W|-V-W}^G\ar[r]^{\res^{G}_{G_W}}&H\M_{|V|+|W|-V-W}^{G_W}\ar[r]^{\res^{G_W}_{G_V}}&H\M_{|V|+|W|-V-W}^{G_V}.
}
\]
By Theorem \ref{keythm}, the above diagram is
\[
\xymatrix{
\M(G_W)^{G}\ar[r]^{i}\ar[d]&\M(G_W)^{}\ar[r]^{\res^{G_W}_{G_V}}\ar[d]&\M(G_V)\ar[d]^{\cong}\\
\M(G_V)^{G}\ar[r]^{i}&\M(G_V)^{G_W}\ar[r]^{i}&\M(G_V).
}
\]
All maps labeled $i$ are inclusion of fixed points, and we deduce the first vertical map is induced from $\res^{G_W}_{G_V}$. This proves $(1)$. An analogous argument will prove $(2)$.

Then we see that the following isomorphic diagrams commute
\[
\xymatrix{
H\M_0^G\ar[r]^-{u_W}\ar[d]^{u_V}&H\M_{|W|-W}^G\ar[d]^{u_V}&\M(G)\ar[r]^{\res^{G}_{G_W}}\ar[d]^{\res^{G}_{G_V}}&\M(G_W)^G\ar[d]^{\res^{G_W}_{G_V}}\\
H\M_{|V|-V}^G\ar[r]^-{u_W}&H\M_{|V|+|W|-V-W}^G,&\M(G_V)^G\ar[r]^1&\M(G_V)^G.
}
\]
When $\M=\A$, the two images of $1\in H\A_0^G$ in $H\A_{|V|+|W|-V-W}^G$ agree, which proves $(3)$.
\end{proof}

\section{Vanishing, \texorpdfstring{$a,u$}{}-isomorphism regions and the gold relation}
In this section we determine the parts of the $RO(G)$-graded homotopy Mackey functors $H\M_{\star}$ where multiplication by different Euler and orientation classes are isomorphism. We also study the parts when they vanish and deduce a gold relation for Green functors $\R$, generalizing the original gold relation of \cite{HHRb}.

\subsection{\texorpdfstring{$a_{\lambda_s},u_{\lambda_s}$}{}-isomorphism regions}
In this subsection we study the $RO(G)$-graded regions where multiplications by $a_{\lambda_s},\aal$ and $u_{\lambda_s}$ are isomorphisms respectively. For simplicity, we write $\lambda=\lambda_k$, as it will play a very important role. Recall that we are working over $G=C_{2^k}$, thus $\lambda$ is the faithful representation.

We start with definitions of Mackey functors which will be used for the description of the homology of a unit representation sphere.
\begin{definition}[\cite{TW95}]\label{restricted Mackey functor}
    For $H\subset G=C_{2^k}$, let $\M\downarrow^G_H$ be the $H$-Mackey functor defined by
    \[
    \M\downarrow^G_H(H/K)=M(G/K), K\subset H,
    \]
    with restrictions and transfers induced from that of $\M$. This is the \emph{restricted} Mackey functor. The restriction functor $\downarrow^G_H:Mack_G\to Mack_H$ is the right adjoint of the induction functor $\uparrow^G_H$.
\end{definition}

\begin{definition}[\cite{TW95}]
\label{def FP FQ}
    Let $G=C_{2^k}$ and let $M$ be a $\mathbb{Z}[G]$-module. We define:
    \begin{enumerate}
        \item The \emph{fixed-point Mackey functor} $FP(M)$ by
        \[
    FP(M)(G/H)=M^H.
    \]
    Since restrictions and transfers are transitive, we specify them for $H=C_{2^i}=\langle\gamma^{2^{k-i}}\rangle$ and $K=C_{2^{i-1}}$ for $1\leq i\leq k$. We have $\res^H_K=1$ and $\tr^H_K=1+\gamma^{2^{k-i}}$.
    \item The \emph{fixed-quotient Mackey functor} $FQ(M)$ by
    \[
    FQ(M)(G/H)=M/H.
    \]
    Since restrictions and transfers are transitive, we specify them for $H=C_{2^i}=\langle\gamma^{2^{k-i}}\rangle$ and $K=C_{2^{i-1}}$ for $1\leq i\leq k$. We have $\res^H_K=1+\gamma^{2^{k-i}}$ and $\tr^H_K=1$.
    \end{enumerate}
Recall that we have the forgetful functor $U:Mack_G\to \mathbb{Z}[G]-mod,\M\mapsto \M(e)$. Then $FP$ is the right adjoint of $U$ and $FQ$ is the left adjoint.

\end{definition}
\begin{definition}
\label{TFP TFQ}
    Let $G=C_{2^k}$ and let $M$ be a $\mathbb{Z}[G]$-module. Define a new $\mathbb{Z}[G]$-module structure on $M$, denoted by $\widetilde{M}$, via the rule $\gamma \ast m\coloneqq -\gamma\cdot m$, where the right-hand-side action is the original $G$-action on $M$. 
    Then we define:
    \begin{enumerate}
        \item The \emph{twisted fixed-point Mackey functor} $TFP(M)$ by 
        \[
        TFP(M)= FP(\widetilde{M}).
        \]
    \item The \emph{twisted fixed-quotient Mackey functor} $TFQ(M)$ by 
    \[
    TFQ(M)= FQ(\widetilde{M}).
    \]
    \end{enumerate}
    Note that $\gamma^2\ast m=\gamma^2 \cdot m$. This implies $TFP(M)\downarrow^G_{G'}=FP(M)\downarrow^G_{G'}$ and $TFQ(M)\downarrow^G_{G'}=FQ(M)\downarrow^G_{G'}$. The top two levels of $TFP(M)$ and $TFQ(M)$ look like
    \[
\begin{tikzcd}
    _{(1+\gamma)}M\dar[bend right = 20, "1"'] &\text{and}&\quad M/(1+\gamma)\dar[bend right = 20, "1-\gamma"'] \\
    _{(1-\gamma^2)}M\uar[bend right = 20,"1-\gamma"'] &&\quad M/{(1-\gamma^2)}.\uar[bend right = 20,"1"']
\end{tikzcd}
    \]
\end{definition}

\begin{remark}
\leavevmode
\begin{enumerate}
    \item The action on $\widetilde{M}$ is the same as the diagonal action on $\tilde{\ZZ}\otimes_{\ZZ} M$ where $\widetilde{\ZZ}$ is the sign representation of $G$. To avoid confusion, if $\M$ is a Mackey functor, we will always use $\gamma \cdot m$ to denote the canonical Weyl action for $m\in \M(H),H\subset G$.
    \item The adjective 'twisted' comes from that they arise from $H\M_U^G$ when $U$ is non-orientable. The top differential in this case is $1+\gamma$ on $\M(e)$, which corresponds to $1-\gamma$ on $\widetilde{\M(e)}$.
\end{enumerate}
\end{remark}

Consider now the following cofiber sequence:
\[
S(V)_+\to S^0\to S^V.
\]
We can use this cofiber sequence to compute the groups $H\M^G_{nV}$ for an irreducible actual $G$-representation and $n\in\mathbb{Z}$ and the action of the element $a_\lambda$ on $H\M^G_\star$. We will firstly need the following lemma.
\begin{lemma}
\label{lem homology Slambda}
The $RO(G)$-graded homotopy Mackey functor $H\M_{V}(S(\lambda)_+)$ is given as follows:
\[
H\M_{V}(S(\lambda)_+)=
\left\{
\begin{array}{ll}
FQ(\M(e))&\textrm{if }|V|=0\textrm{ and }V\textrm{is orientable}\\
TFQ(\M(e))&\textrm{if }|V|=0\textrm{ and }V\textrm{is non-orientable}\\
FP(\M(e))&\textrm{if }|V|=1\textrm{ and }V\textrm{is orientable}\\
TFP(\M(e))&\textrm{if }|V|=1\textrm{ and }V\textrm{is non-orientable}\\
0&\textrm{else.}
\end{array}
\right.
\]

\end{lemma}
\begin{proof}
We first compute the $G/G$-level abelian group and claim it is as follows:
\[
H\M^G_{V}(S(\lambda)_+)=
\left\{
\begin{array}{ll}
\M(e)/(1-\gamma)&\textrm{if }|V|=0\textrm{ and }V\textrm{is orientable}\\
\M(e)/(1+\gamma)&\textrm{if }|V|=0\textrm{ and }V\textrm{is non-orientable}\\
\M(e)^G&\textrm{if }|V|=1\textrm{ and }V\textrm{is orientable}\\
_{(1+\gamma)}\M(e)&\textrm{if }|V|=1\textrm{ and }V\textrm{is non-orientable}\\
0&\textrm{else.}
\end{array}
\right.
\]

The unit sphere $S(\lambda)$ is the homotopy cofiber in the following sequence:
\[
\begin{tikzcd}
G/e\rar["1-\gamma"]&G/e\rar&S(\lambda).
\end{tikzcd}
\]
By the long exact sequence in homotopy associated to this cofiber sequence we get:
\[
\begin{tikzcd}[row sep=small]
H\M^G_{V+1}(G/e_+)\rar&
H\M^G_{V+1}(S(\lambda)_+)\rar&
H\M^G_{V}(G/e_+)\rar["1-\gamma"]&
H\M^G_{V}(G/e_+)\rar&\phantom{.}\\
H\M^G_{V}(S(\lambda)_+)\rar&
H\M^G_{V-1}(G/e_+).&&&
\end{tikzcd}
\]
For any $G$-representation $W$ we have that $H\M^G_{W}(G/e_+)\cong H\M^e_{|W|}$. In particular, it is zero if $|W|\neq 0$ and equal to $M(e)$ if $|W|=0$ as abelian groups. As $\mathbb{Z}[G]$-modules, the action depends on the orientability of $W$. If $W$ is such that $|W|=0$ and orientable, then $H\M^G_{W}(G/e_+)=\M(e)$; if $W$ is non-orientable, then $H\M^G_{W}(G/e_+)=\widetilde{\M(e)}$. We can see this fact by looking at the shearing isomorphisms (see Proposition \ref{shearing})
\[
\xymatrix{
G/e_+\wedge S^W\ar[d]^{\cong}\ar[r]^{\gamma\wedge 1} &G/e_+\wedge S^W\ar[d]^{\cong}&(g,a)\ar@{|->}[r]\ar@{|->}[d]&(\gamma g,a)\ar@{|->}[d]\\
G_+\wedge _e S^{|W|}\ar[r]^{}&G_+\wedge_e S^{|W|},&(g,g^{-1}a)\ar@{|->}[r]&(\gamma g,g^{-1}\gamma^{-1}a).
}
\]

Therefore in the sequence above the third and fourth arguments are zero unless $|V|=0$. In the latter case we have that the outer groups are equal to zero and the long exact sequence above takes the form:
\[
\begin{tikzcd}[column sep=small]
0\rar&
H\M^G_{V+1}(S(\lambda)_+)\rar&
\M(e)\rar["1-\gamma"]&
\M(e)\rar&
H\M^G_{V}(S(\lambda)_+)\rar&
0
\end{tikzcd}
\] or
\[
\begin{tikzcd}[column sep=small]
0\rar&
H\M^G_{V+1}(S(\lambda)_+)\rar&
\widetilde{\M(e)}\rar["1-\gamma"]&
\widetilde{\M(e)}\rar&
H\M^G_{V}(S(\lambda)_+)\rar&
0,
\end{tikzcd}
\]
depending on the orientability of $V$. The conclusion on the abelian group structure follows. The Mackey functor structure can be computed as in Remark \ref{remark Mackey functor from cellular chain}.
\end{proof}

\begin{remark}
Using the same cofiber sequence we can compute the $H\M$-cohomology Mackey functor of $S(\lambda)_+$ and obtain that $H\M^{-V+1}(S(\lambda)_+)\cong H\M_{V}(S(\lambda)_+)$ as Mackey functors. This comes from the fact that $S(\lambda)_+\simeq S^1\wedge D(S(\lambda)_+)$.
\end{remark}

\begin{proposition}[$a_{\lambda}$-isomorphism region]
\label{prop alambda iso}
Let $V$ be a $G$-representation such that $|V|\neq 0,1,2$. Then the multiplication map $a_\lambda\colon H\M_V\to H\M_{V-\lambda}$ is an isomorphism of Mackey functors.
\end{proposition}
\begin{proof}
Consider the cofiber sequence $S(\lambda)_+\to S^0\to S^{\lambda}$. Taking $V$-th homology we obtain
\[
\begin{tikzcd}
H\M^G_{V}(S(\lambda)_+)\rar& H\M^G_V\rar&H\M^G_{V-\lambda}\rar& H\M^G_{V-1}(S(\lambda)_+).
\end{tikzcd}
\]
If $|V|\neq 0,1,2$ the outer groups are zero by Lemma \ref{lem homology Slambda}. This is also true when we replace $G$ by any subgroup of $G$. Thus the claim follows.
\end{proof}

\begin{remark}
    Note that traditionally the term `$a_{\lambda}$-periodic' classes is used for classes that support an infinite $a_{\lambda}$-tower. From the above Proposition, we know non-trivial classes of total dimensions $\leq -1$ are $a_{\lambda}$-periodic.
\end{remark}

\begin{proposition}
\label{prop action of alambda}
Let $V$ be a $G$-representation such that $|V|=0$ and recall that $\res(V)$ and $\tr(V)$ are structure maps of $H\M_V$. Then:
\begin{enumerate}
\item $H\M^G_{V-\lambda}\cong \coker(\tr(V)^G_e)$ and the map $a_\lambda\colon H\M^G_V\to H\M^G_{V-\lambda}$ is the projection onto the cokernel.
\item $H\M^G_{V+\lambda}\cong \ker(\res(V)^G_e)$ and the map $a_\lambda\colon H\M^G_{V+\lambda}\to H\M^G_{V}$ is the inclusion of the kernel.
\end{enumerate}
The Mackey functor structures on $H\M^G_{V-\lambda}$ and $H\M^G_{V+\lambda}$ are induced from that of $H\M_V^G$.
\end{proposition}
\begin{proof}
For point (1), consider the long exact sequence obtained by smashing cofiber sequence $S(\lambda)_+\to S^0\to S^{\lambda}$ with $S^V\wedge H\M$:
\[
\begin{tikzcd}
H\M^G_{V}(S(\lambda)_+)\rar& H\M^G_V\rar&H\M^G_{V-\lambda}\rar& H\M^G_{V-1}(S(\lambda)_+).
\end{tikzcd}
\]
Since $|V|=0$, by Lemma \ref{lem homology Slambda} the rightmost group is 0. To understand the first map in the sequence above, consider the following commutative diagram of cofiber sequences:
\[
\xymatrix{
G/e_+\ar[r]\ar[d]&S^0\ar[r]\ar[d]& S^{\lambda}_{(1)}\ar[d]\\
S(\lambda)_+\ar[r]&S^0\ar[r]&S^{\lambda}
}
\]
where $S^\lambda_{(1)}$ denotes the $1$-skeleton of $S^\lambda$. The left vertical map is the inclusion of the bottom cell. Taking $V$-th homology we obtain the following commutative diagram:
\[
\begin{tikzcd}
M(e)\rar\dar&H\M^G_V\rar\dar["="]&H\M^G_V(S^\lambda_{(1)})\rar\dar&0\\
M(e)/(1-\gamma)\rar& H\M^G_V\rar&H\M^G_{V-\lambda}\rar& 0.
\end{tikzcd}
\]
The first top horizontal map exhibits the transfer in the Mackey functor $H\M^\bullet_V$. Therefore the bottom left map is the map on $M(e)/1-\gamma$ induced by $\tr(V)$. Thus the claim follows.

Point (2) comes from a similar analysis by taking $V$-th cohomology.

For the Mackey functor structure in (1), by induction, we only need to consider the restriction and transfer between adjacent subgroups. In the following diagram
\[
\xymatrix{
H\M_V^G\ar[r]\ar@/^/[d]^{\res^G_{G'}}&H\M_V^G/{\tr(V)^G_e}\ar@/^/[d]^{\overline{\res^G_{G'}}}\\
H\M_V^{G'}\ar[r]\ar@/^/[u]^{\tr^G_{G'}}&H\M_V^{G'}/{\tr(V)^{G'}_e}\ar@/^/[u]^{\overline{\tr^G_{G'}}}}
\]
both horizontal arrows are the natural surjections, and the result follows. The Mackey functor structure in (2) are completely analogous.
\end{proof}

Combining Propositions \ref{prop alambda iso} and \ref{prop action of alambda} we can deduce groups $H\M^G_{nV}$ for $V$ orientable and irreducible.

\begin{proposition}
\label{prop hmnV}

Let $V$ be an irreducible actual representation. Then we have that
\[
H\M^G_{nV}=\left\{
\begin{array}{ll}
\M(G) & n=0\\
\coker(\tr_{G_V}^G)&n<0\\
\ker(\res^G_{G_V})&n>0.
\end{array}
\right.
\]
Moreover,  we have that:
\begin{itemize}
\item the map $a_V\colon H\M^G_0\to H\M^G_{-V}$ is the projection onto the cokernel;
\item the map $a_V\colon H\M^G_{V}\to H\M^G_0$ is the inclusion of the kernel;
\item the multiplication $a_V\colon H\M^G_{nV}\to H\M^G_{(n-1)V}$ is an isomorphism for $n\in\mathbb{Z}$ such that $n\neq 0,1$.
\end{itemize}
\end{proposition}
\begin{proof}
    We firstly consider the case $V=\lambda$, the faithful representation. We then have that $H\M^G_0=\M(G)$ by the defining property of Eilenberg-MacLane spectra. For $n<0$, the formula follows from Proposition \ref{prop action of alambda} in case of $n=-1$ and from Proposition \ref{prop alambda iso} for $n<-1$. The case $n>0$ follows analogously.

    Now assume that $V$ is an irreducible representation. Then we have that $H\M^G_{nV}\cong H(\M^{\sharp G_V})^{G/G_V}_{nV}$. Assume that $V\neq \alpha$. The representation $V$ is $G/G_V$-faithful, so by the previous paragraph we get that
\[
H(\M^{\sharp G_V})^{G/G_V}_{nV}=\left\{
\begin{array}{ll}
\M(G) & n=0\\
\coker(\tr(\M^{\sharp G_V})_e^{G/G_V})&n<0\\
\ker(\res(\M^{\sharp G_V})_e^{G/G_V})&n>0.
\end{array}
\right.
\]
However, since the transfers and restrictions in the Mackey functor $\M^{\sharp G_V}$ are defined as transfers and restrictions in $\M$, so we obtain the abelian group structures as stated. For the multiplications by $a_V$, as above, we can assume $V=\lambda$, the faithful representation. Consider the following cellular filtration of $S^{m\lambda},m\geq 1$:
\begin{equation*}
\begin{tikzcd}[column sep = tiny]
    S^0\ar[rr] &&S^{m\lambda}_{(1)}\ar[rr]\ar[ld]&&S^{m\lambda}_{(2)}\ar[ld]\ar[rr]&&S^{m\lambda}_{(3)}\ar[rr]\ar[ld]&&\cdots\ar[ld]\\
    &{G/e}_+\wedge S^1\ar[ul,dashed] &&{G/e}_+\wedge S^2\ar[ul,dashed] &&{G/e}_+\wedge S^3\ar[ul,dashed] &&{G/e}_+\wedge S^4\ar[ul,dashed]
\end{tikzcd}
\end{equation*}
Smash the diagram with $H\M$ and take $\pi_0$. Since $\pi_{0}^G G/e_+\wedge S^i\wedge H\M=0$ for $i\neq 0$, we see that the first horizontal map induce the projection onto the cokernel of $\tr^G_e$, and the other horizontal maps induce isomorphisms. This proves the claim about $a_V:H\M^G_{nV}\to H\M^G_{(n-1)V}$ when $n\leq 0$. For $n>0$, we can dualize the cellular structure.

The case $V=\alpha$ follows from the similar considerations.
\end{proof}

Using the multiplication by the element $a_\lambda$, we can prove the following vanishing result.
\begin{lemma}[Vanishing]
\label{lem vanishing}
Let $V=a_0+a_1\alpha+\sum_{i=2}^ka_i\lambda_i$ be such that one of the following conditions is satisfied:
\begin{enumerate}
\item $a_0>0$, $a_0+a_1>0$ and $a_0+a_1+\sum_{i=2}^ja_i>0$ for all $2\leq j\leq k$, or
\item $a_0<0$, $a_0+a_1<0$ and $a_0+a_1+\sum_{i=2}^ja_i<0$ for all $2\leq j\leq k$.
\end{enumerate}
Then $H\M_V=0$.
\end{lemma}
\begin{proof}
We will proceed by induction on the order of $G$. The claim is true for the trivial group. Now assume that the statement is correct for $G'=C_{2^{k-1}}$. Thus the vanishing is true for all subgroup level of the Mackey functor. Consider the first case with all sums simultaneously greater than $0$. We then have that
\[
H\M^G_{a_0+a_1\alpha+a_2\lambda_2+\ldots +a_{k-1}\lambda_{k-1}}\cong (H\M^{\sharp C_2})^{G/{C_2}}_{a_0+a_1\alpha+a_2\lambda_2+\ldots +a_{k-1}\lambda_{k-1}},
\]
and the latter is $0$ by the inductive assumption (note that the last index in the sum missing). By Proposition \ref{prop alambda iso} the map
\[
a_{\lambda_k}^{a_{k}}\colon H\M^G_V\to H\M^G_{a_0+a_1\alpha+a_2\lambda_2+\ldots +a_{k-1}\lambda_{k-1}}
\]
is an isomorphism, because the target has total dimension greater than $0$. Thus the claim follows. The proof of the second case is completely analogous.
\end{proof}

Using  Vanishing Lemma \ref{lem vanishing} we can describe $a_{\lambda_i}$-isomorphism regions for $2\leq i\leq k$.

\begin{corollary}[$a_{\lambda_s}$-isomorphism region]
For $C_{p^k}$ with $p$ odd, let $V=a_0+\sum_{i=1}^ka_i\lambda_i$. For each $s$ such that $ 1\leq s\leq k$, multiplication by $a_{\lambda_s}\colon H\M_{V}\to H\M_{V-\lambda_s}$ is an isomorphism of Mackey functors when one of the following conditions is satisfied:
\begin{enumerate}
    \item $a_0+2\sum_{i=1}^{j}a_i>2$ for all $s\leq j\leq k$, or
    \item $a_0+2\sum_{i=1}^{j}a_i<0$ for all $s\leq j\leq k$.
\end{enumerate}
For $C_{2^k}$, let $V=a_0+a_1\alpha+\sum_{i=2}^ka_i\lambda_i$. For each $s$ such that $2\leq s\leq k$, multiplication by $a_{\lambda_s}\colon H\M_V\to H\M_{V-\lambda_s}$ is an isomorphism of Mackey functors when one of the following conditions is satisfied:
\begin{enumerate}
    \item $a_0+a_1+2\sum_{i=2}^{j}a_i>2$ for all $s\leq j\leq k$, or
    \item $a_0+a_1+2\sum_{i=2}^{j}a_i<0$ for all $s\leq j\leq k$.
\end{enumerate}
Multiplication by $\aal\colon H\M_{V}\to H\M_{V-\alpha}$ is an isomorphism if one of the following conditions is satisfied:
\begin{enumerate}
    \item $a_0+a_1>0,a_0+a_1+2\sum_{i=2}^{j}a_i>1$ for all $2\leq j\leq k$, or
    \item $a_0+a_1<0,a_0+a_1+2\sum_{i=2}^{j}a_i<0$ for all $2\leq j\leq k$.
\end{enumerate}
\end{corollary}
\begin{proof}
The proof for $p$ odd and $p=2$ are analogous. We will assume $p=2$, and use the notation $\lambda_1=2\alpha$. For each $2\leq s\leq k$, the sphere $S(\lambda_s)$ is a two-cell complex, and we have a cofiber sequence
\[
{C_{2^k}/C_{2^{k-s}}}_+\to S(\lambda_s)_+\to {C_{2^k}/C_{2^{k-s}}}_+\wedge S^1.
\]
We get the following long exact sequence of Mackey functors
\[
\cdots\to H\M_{V}(C_{2^k}/C_{2^{k-s}})\to H\M_V(S(\lambda_s))\to H\M_{V-1}(C_{2^k}/C_{2^{k-s}})\to\cdots.
\]
The two Mackey functors on the outside only cares about the $C_{2^{k-s}}$-representation $i^*_{C_{2^{k-s}}}(V)=a_0+a_1+2a_2+\cdots+2a_s+a_{s+1}\lambda_1+\cdots+a_k\lambda_{k-s}$. By the previous lemma, we deduce that $\underline{H_V(S(\lambda_s))}=0$ when one of the following is satisfied
\begin{enumerate}
    \item $a_0+a_1+2\sum_{i=2}^{j}a_i>1$ for all $s\leq j\leq k$, or
    \item $a_0+a_1+2\sum_{i=2}^{j}a_i<0$ for all $s\leq j\leq k$.
\end{enumerate}
Then the cofiber sequence
\[
S(\lambda_s)_+\to S^0\to S^{\lambda_s}
\]
provides us with the long exact sequence
\[
\cdots\to H\M_V(S(\lambda_s))\to H\M_V\xrightarrow{a_{\lambda_s}} H\M_{V-\lambda_s}\to H\M_{V-1}(S(\lambda_s))\cdots.
\]
$a_{\lambda_s}$ is an isomorphism when the two outside Mackey functors vanish. And we have finished the proof for $a_{\lambda_s}$. For $\aal$, the only difference is that the sphere $S(\alpha)=C_{2^k}/C_{2^{k-1}}$ is a one-cell complex. All the proofs are along the same line.
\end{proof}

Using the vanishing property, we can also deduce when multiplication by $u_{\lambda_s}$ is an isomorphism. From Theorem \ref{keythm} we deduce
\begin{equation}
\label{Cui}
    \pi_*^G H\M\wedge S^{\lambda_i-2}=H_{2+*}^G(S^{\lambda_i};\M)=\begin{cases}
  \M(C_{2^{k-i}})^G,  & *=0; \\
  \ker(\tr^G_{C_{2^{k-i}}})/(1-\gamma), & *=-1; \\
  \coker(\tr^G_{C_{2^{k-i}}}), & *=-2;\\
  0 & \text{otherwise}.
\end{cases}
\end{equation}
Then we look at the cofiber sequence
\begin{equation}
    H\M\xrightarrow{u_{\lambda_i}}H\M\wedge S^{\lambda_i-2}\to C_{u_{\lambda_i}}.
\end{equation}
In general, the cofiber $C_{u_{\lambda_i}}$ is not an Eilenberg-MacLane spectrum (it is one in the special case of $\M=\Z$), but it has a finite Postnikov tower, which enables us to extend the vanishing results to this case. Before that, first let us notice that the map $u_{\lambda_i}$ is an $C_{2^{k-i}}$-equivariant equivalence, thus $i^*_{C_{2^{k-i}}}C_{u_{\lambda_i}}\simeq *$. Since $S^{\lambda_s}$ is built from $S^0$ out of cells of orbit type $C_{2^k}/C_{2^{k-s}}$, we deduce

\begin{lemma}
The spectrum $C_{u_{\lambda_i}}$ is $a_{\lambda_i},a_{\lambda_{i+1}},\cdots, a_{\lambda_k}$-local.
\end{lemma}

\begin{corollary}[$u_{\lambda_s}$-isomorphism region]
For $p$ odd, let $V=a_0+\sum_{i=1}^ka_i\lambda_i$. For each $s$ such that $1\leq s\leq k$, multiplication by $u_{\lambda_s}\colon H\M_{V}\to H\M_{V+2-\lambda_s}$ is an isomorphism of Mackey functors when one of the following conditions is satisfied:
\begin{enumerate}
    \item $a_0>0$ and $a_0+2\sum_{i=1}^{j}a_i>0$ for all $1\leq j\leq s-1$, or
    \item $a_0<-3$ and $a_0+2\sum_{i=1}^{j}a_i<-3$ for all $1\leq j\leq s-1$.
\end{enumerate}
If $p=2$, let $V=a_0+a_1\alpha+\sum_{i=2}^ka_i\lambda_i$. For each $s$ such that $2\leq s\leq k$, multiplication by $u_{\lambda_s}\colon H\M_V\to H\M_{V+2-\lambda_s}$ is an isomorphism of Mackey functors when one of the following conditions is satisfied:
\begin{enumerate}
    \item $a_0>0$ and $a_0+a_1>0,a_0+a_1+2\sum_{i=2}^{j}a_i>0$ for all $2\leq j\leq s-1$, or
    \item $a_0<-3$ and $a_0+a_1<-3,a_0+a_1+2\sum_{i=2}^{j}a_i<-3$ for all $2\leq j\leq s-1$.
\end{enumerate}
Multiplication by $\uta\colon H\M_V\to H\M_{V+2-2\alpha}$ is an isomorphism when $a_0>0$ or $a_0<-3$.
\end{corollary}

\begin{proof}
Again we assume $p=2$ and the odd primary case is analogous. We also assume $2\leq s\leq k$ since the $u_{2\alpha}$ case is proved in the same way. Consider the long exact sequence
\begin{equation*}
    \underline{\pi}_{V+1}C_{u_{\lambda_s}}\to H\M_{V}\xrightarrow{u_{\lambda_s}} H\M_{V+2-\lambda_s}\to \underline{\pi}_{V}C_{u_{\lambda_s}}.
\end{equation*}
We want to find the range where the two outside Mackey functors vanish. By $a_{\lambda_s},\cdots,a_{\lambda_k}$-periodicity of $\pi_{\star}C_{u_{\lambda_s}}$, we can assume $V=x+y\alpha+\sum_{i=2}^{s-1}a_i\lambda_i$. From the computation \eqref{Cui}, we know $\pi_{*}C_{u_{\lambda_s}}$ is concentrated at $*=-2,-1,0$, and we get its Postnikov tower
\[
\xymatrix{
C_{u_{\lambda_i}}=P^0\ar[d]&\ar[l]H\underline{M_3}\\
P^{-1}\ar[d]&\ar[l]\Sigma^{-1} H\underline{M_2}\\
P^{-2}&\ar[l]^{\simeq}\Sigma^{-2} H\underline{M_1}
}
\]
where $\underline{M_i},i=1,2,3$ are some Mackey functors. Then we get a strongly convergent spectral sequence
\[
\underline{E^1}_{V,s}=\underline{\pi}_V \bar{P}^s\Rightarrow \underline{\pi}_V C_{u_{\lambda_s}}
\]
where $\bar{P}^s$ is the fiber of $P^s\to P^{s-1}$. The conditions stated are the conditions for $\underline{E^1}_{V,s}=0$ and $\underline{E^1}_{V+1,s}=0$.
\end{proof}

\begin{remark}
For $\M=\Z$, from the formula (\ref{Cui}) and the fact that $u_{\lambda_s}:\Z=H\Z_0\xrightarrow{\cong} H\Z_{2-\lambda_s}=\Z$ we know $C_{u_{\lambda_s}}\simeq \Sigma^{-2}H\underline{{M_1}}$. We can get a finer condition than the one stated in the theorem, which is left to the reader.
\end{remark}

\subsection{\texorpdfstring{$au$}{}-relation (a.k.a gold relation)}
In this subsection, we generalize the gold relation \cite[Lemma 3.6]{HHRb} for $\R=\Z$ to the general case.
\begin{proposition}[$au$-relation]
\label{prop au relation}
Let $V,W$ be chosen from $\lambda_i$ for $1\leq i\leq k$ such that $G_V\subset G_W$. Let $\R$ be a Green functor. Then there is the following relation in $H\R_{\star}^G$
\[
a_V\cdot\tr_{G_V}^{G_W}(1)u_W=a_W\cdot u_V.
\]
\end{proposition}

Note since $H\R^G_{2-W}=\R(G_W)^G$ by Proposition \ref{prop mackey uv}, and $\tr^{G_W}_{G_V}(1)\in \R(G_W)^G$, the class $\tr_{G_V}^{G_W}(1)u_W\in H\R^G_{2-W}$ makes sense.

\begin{lemma}
\label{lemma hm2Vlambda}
Let $W$ be one of $\lambda_i$ for $1\leq i\leq k$. Then $H\M^G_{2-W-\lambda}\cong \frac{\M(G_W)^G}{\zeta_{\log_2|G/G_W|}\tr^{G_W}_e(\M(e))}$.
\end{lemma}
\begin{proof}
Follows from Theorem \ref{keythm}.
\end{proof}

\begin{lemma}
\label{lemma au relation free representation}
Let  $W$ be  one of $\lambda_i,1\leq i\leq k$, then we have the following relation in $H\R_{\star}^G$
\[
a_\lambda\cdot \tr_e^{G_W}(1)u_W=a_W\cdot u_\lambda.
\] 
\end{lemma}
\begin{proof}
We only need to prove for the universal case $\R=\A$. We claim that the following commutative diagrams are isomorphic
\[
\xymatrix{
H\A_0^G\ar[r]^{a_{\lambda}}\ar[d]^{u_{\lambda}} &H\A_{-\lambda}^G\ar[d]^{u_W} &\A(G)\ar[d]^{\res^G_e}\ar[r]^{\text{proj}}&\A(G)/{\tr^G_e}\ar[d]^{\res^G_{G_W}}\\
H\A_{2-\lambda}^G\ar[r]^-{a_W}&H\A_{2-\lambda-W}^G,&\A(e)\ar[r]^-{\tr^{G_W}_e}&\A(G_W)/|G/G_W|\tr^{G_W}_e.
}
\]
The corresponding abelian groups are isomorphic by previous computations, and that $\A(H)$ has trivial Weyl action for each $H\subset G$. That $a_{\lambda}$ is projection and $u_{\lambda}$ is $\res^G_e$ are proved in Proposition \ref{prop action of alambda} and Remark \ref{remark action of uv} respectively.

To show that $u_W$ is induced from $\res^G_{G_W}$, we look at the following isomorphic commutative diagrams
\[
\xymatrix{
H\A_{-\lambda}^G\ar[r]^{\res^G_{G_W}}\ar[d]^{u_{W}} &H\A_{-\lambda}^{G_W}\ar[d]^{u_W} &\A(G)/\tr^G_e\ar[d]^{\res^G_{G_W}}\ar[r]^{\res^G_{G_W}}&\A(G_W)/{\tr^{G_W}_e}\ar[d]^{1}\\
H\A_{2-\lambda-W}^G\ar[r]^{\res^G_{G_W}}&H\A_{2-\lambda-W}^{G_W},&\A(G_W)/{|G/G_W|\tr_e^{G_W}}\ar[r]^-{f}&\A(G_W)/\tr^{G_W}_e.
}
\]
The map $f$ is induced from the identity of $\A(G_W)$ thus send $\bar{x}$ to $\bar{x}$, and we deduce that $u_W$ is induced from $\res^G_{G_W}$.

We are left to show that $a_W$ is induced from $\tr_e^{G_W}$. We consider the following filtration of $S^{\lambda}\wedge S^W$

\[
\begin{tikzcd}[column sep = tiny]S^0\ar[rr]&&S^{\lambda}_{(1)}\ar[rr]\ar[ld]&&S^{\lambda}\ar[rr]\ar[ld]&[-7pt]&[-7pt]S^\lambda\wedge S^W_{(1)}\ar[rr]\ar[ld]&[-7pt]&[-7pt]S^{\lambda}\wedge S^W\ar[ld]\\
&G/e_+\wedge S^1\ar[ul,dashed] &&G/e_+\wedge S^2\ar[ul,dashed]&&{G/G_W}_+\wedge S^{\lambda}\wedge S^1\ar[ul,dashed]&&{G/G_W}_+\wedge S^{\lambda}\wedge S^2\ar[ul,dashed].
\end{tikzcd}
\]

The map $a_W:S^{\lambda}\to S^{\lambda}\wedge S^W$ sits in the first row. By running a spectral sequence, we can see that the classes that contribute to $\pi_2^G (S^{\lambda}\wedge S^W\wedge H\A)$ are subquotients of 
$\pi_2^G(G/e_+\wedge S^2\wedge H\A)$, $\pi_2^G({G/G_W}_+\wedge S^{\lambda}\wedge S^1\wedge H\A)$ and $\pi_2^G({G/G_W}_+\wedge S^{\lambda}\wedge S^2\wedge H\A)$.
We have
\begin{equation*}
    \begin{aligned}
        &\pi_2^G(G/e_+\wedge S^2\wedge H\A)=\A(e),\\
        &\pi_2^G({G/G_W}_+\wedge S^{\lambda}\wedge S^1\wedge H\A)=H\A^{G_W}_1(S^{\lambda})={\ker(\tr_e^{G_W})}/{(1-\gamma)}=0,\\
        &\pi_2^G({G/G_W}_+\wedge S^{\lambda}\wedge S^2\wedge H\A)=H\A^{G_W}_0(S^{\lambda})=\A(G_W)/\tr^{G_W}_e.
    \end{aligned}
\end{equation*}
We already know that
\[
\pi_2^G (S^{\lambda}\wedge S^W\wedge H\A)=H\A^G_{2-W-\lambda}\cong \frac{\A(G_W)}{|G/G_W|\tr^{G_W}_e}\cong \A(G_W)/x\oplus \mathbb{Z}/{|G/G_W|}\langle x\rangle
\]
by the previous lemma, where $x$ is represented by the free $G_W$-set $G_W/e=\tr_e^{G_W}(1)$. The differential
\[
d_1:\A(e)=\pi_2^G({G/G_W}_+\wedge S^{\lambda}\wedge H\A)\to \pi_2(G/e_+\wedge S^2\wedge H\A)=\A(e)
\]
must be $|G/G_W|$, and we have the following extension
\[
0\to \A(e)/{|G/G_W|}\xrightarrow{h} H\A_{2-\lambda-W}^G\to \A(G_W)/\tr^{G_W}_e\to 0
\]
where $h$ is induced from $\tr^{G_W}_e$. Thus $a_W$ is induced from $\tr^{G_W}_e$.

In $\A(G_W)/x\oplus \mathbb{Z}/{|G/G_W|}\langle x\rangle$ we have that $a_{\lambda}u_W=(1,0)$ and $a_{W}u_{\lambda}=(0,x)$, where $1$ is $G_W/{G_W}$. We deduce
\[
a_Wu_\lambda=a_\lambda\cdot  xu_W=a_\lambda\cdot \tr^{G_W}_e(1) u_W.
\]

\end{proof}
\begin{proof}[Proof of Proposition \ref{prop au relation}]
Let $V,W$ be chosen from $\lambda_i,1\leq i\leq k$ with $e\neq G_V\subset G_W$. Note $\A^{\sharp G_V}$ is a Green functor for $G/G_V$. Regard $V,W$ as $G/G_V$-representations and by abuse of notations, we will still write $V,W$. By induction on $k$, we get
\[
a_V\cdot \tr_{e}^{G_W/{G_V}}(1)u_W=a_Wu_V
\]
in $(H\A^{\sharp G_V})_{\star}$ since $V$ is a faithful ${G/G_V}$-representation. We have
\[
\tr_{e}^{G_W/{G_V}}(1)=\tr_{G_V}^{G_W}(1)
\]
where the transfer on the right is in $\A$. Thus the claim follows.
\end{proof}

\section{Induction theorems}
In this section, we prove our induction theorems. The results fit $H\M_{\star}$ in exact sequences where the rest of the terms are computable in terms of $H\left(\M^{\sharp C_2}\right)_{\star}$ and $H\left(\M\downarrow^G_{C_2}\right)_\star$. Thus we can carry out the computations by induction on the order of $G=C_{2^k}$. The initial input is the structure of $H\M^{C_2}_{\star}$, which is completely determined by the first author in \cite{Sikora22}. 

Recall that Proposition \ref{prop alambda iso} reduces our consideration of $H\M_{V}$ to those $V\in RO(G)$ such that $|V|=2,1,0,-1,-2$. Proposition \ref{prop action of alambda} then tells us how to further reduce to those $V$ with $|V|=1,0,-1$. We first relate the homotopy groups in degrees $V$ with $|V|=\pm 1$.

\begin{proposition}
\label{prop deg pm1}
    Let $V\in RO(G)$ with $|V|=1$. If $V$ is orientable, then we have the following exact sequence of abelian groups
    \[
    0\to \im(\res(V+1-\lambda)_e^G)\to \M(e)^G\to H\M_V^G\xrightarrow{a_{\lambda}}H\M_{V-\lambda}^G\to \M(e)/{(1-\gamma)}\to \im(\tr(V-1)^G_e)\to 0.
    \]
    If $V$ is non-orientable, then we have the following exact sequence of abelian groups
    \[
    0\to \im(\res(V+1-\lambda)_e^G)\to _{\zeta_1}\M(e)\to H\M_V^G\xrightarrow{a_{\lambda}}H\M_{V-\lambda}^G\to \M(e)/{(1+\gamma)}\to \im(\tr(V-1)^G_e)\to 0.
    \]
    Each exact sequence can be upgraded to an exact sequence of Mackey functors.
\end{proposition}
\begin{proof}
    By Lemma \ref{lem homology Slambda}, when $V$ is orientable,
    the long exact sequence on homology of the cofiber sequence
    \[
    S(\lambda)_+\to S^0\xrightarrow{a_{\lambda}}S^{\lambda}
    \]
    gives us
    \begin{align*}
    &H\M_{V+1}^G\xrightarrow{a_{\lambda}}H\M_{V+1-\lambda}^G\xrightarrow{\res(V+1-\lambda)_e^G} \M(e)^G\to H\M_V^G\xrightarrow{a_{\lambda}}H\M_{V-\lambda}^G\to \M(e)/{(1-\gamma)}\xrightarrow{\tr(V-1)_e^G}\\
    &H\M_{V-1}^G\xrightarrow{a_{\lambda}} H\M_{V-1-\lambda}^G.
    \end{align*}
    
    By Proposition \ref{prop action of alambda}, the first $a_{\lambda}$ can be identified with the inclusion of $\ker(\res(V+1-\lambda)^G_e)$, whose cokernel is $\im(\res(V+1-\lambda)_e^G)$. The third $a_{\lambda}$ can be identified with projection to $\coker(\tr(V-1)^G_e)$, whose kernel is $\im(\tr(V-1)^G_e)$. Replacing $G$ by $H\subset G$, we get and exact sequence of Mackey functors.
    
    The case for non-orientable $V$ is completely analogous.
\end{proof}

\begin{remark}
    Proposition \ref{prop deg pm1} tells us that the homotopy Mackey functors with $|V|=1$ and $|V|=-1$ determine each other up to extension if we can figure out the maps in the long exact sequences and understand the $|V|=0$ grading, since both $(V+1-\lambda)$ and $(V-1)$ have underlying degree $0$ there.
\end{remark}

The following two theorems are the main theorems in this section. For simplicity, if $U\in RO(G)$, we use the same notation for its restriction to subgroups. If $G_U\neq e$, we also use the same notation for the corresponding $G/G_U$-representation. For definitions of Mackey functors $FP$, $FQ$ and $TFP$, $TFQ$ see Definitions \ref{def FP FQ} and \ref{TFP TFQ} respectively.

\begin{theorem}[Induction Theorem A]
\label{thm induction 1}
    Let $U=a_0+a_1\alpha+\Sigma^{k}_{i=2}a_i\lambda_i\in RO(G)$ with $|U|=0$. Then
    \begin{enumerate}
        \item When $a_k=0$, $H\M_U$ can be computed from $H\left(\M\downarrow^G_{C_2}\right)_U$ and $H\left(\M^{\sharp C_2}\right)_U$.\\
        \item When $a_k\leq -1$, if $U$ is orientable, we have the following exact sequence of Mackey functors
        \[
        0\to H\M_{U+{\lambda_k}}\xrightarrow{a_{\lambda}}H\M_U\xrightarrow{\res^{(-)}_e}FP(\M(e))\to H\M_{U-1+\lambda_k}.
        \]
        If $U$ is non-orientable, we have the following exact sequence of Mackey functors
        \[
        0\to H\M_{U+{\lambda_k}}\xrightarrow{a_{\lambda}}H\M_U\xrightarrow{\res^{(-)}_e}TFP(\M(e))\to H\M_{U-1+\lambda_k}.
        \]
        Here $\res^{(-)}_e$ is the map induced from $\res^H_e$ on the $G/H$-level. In each sequence, the first and last Mackey functors can be computed from $H\left(\M^{\sharp C_2}\right)_\star$ and $H\left(\M\downarrow^G_{C_2}\right)_\star$.\\
        \item When $a_k\geq 1$, if $U$ is orientable, we have the following exact sequence of Mackey functors
        \[
        H\M_{U-\lambda_k+1}\to FQ(\M(e))\xrightarrow{\tr^{(-)}_e}H\M_U \xrightarrow{a_{\lambda}}H\M_{U-\lambda_k}\to 0.
        \]
        If $U$ is non-orientable, we have the following exact sequence of Mackey functors
        \[
        H\M_{U-\lambda_k+1}\to TFQ(\M(e))\xrightarrow{\tr^{(-)}_e}H\M_U\xrightarrow{a_{\lambda}}H\M_{U-\lambda_k}\to 0.
        \]
        Here $\tr^{(-)}_e$ is the map that is induced from $\tr^H_e$ on the $G/H$-level. In each sequence, the first and last Mackey functors can be computed from $H\left(\M^{\sharp C_2}\right)_\star$ and $H\left(\M\downarrow^G_{C_2}\right)_\star$.
    \end{enumerate}
\end{theorem}
\begin{proof}
    The case $a_k=0$ follows from Corollary \ref{reduction to quotient group}. It describes the restrictions and transfers between subgroups that do not equal to $e$. The rest of the structure can be deduced from $H\left(\M\downarrow^G_{C_2}\right)_U$.

    Consider the cofiber sequence
    \[
    S(\lambda_k)_+\to S^0 \to S^{\lambda_k}.
    \]
    Since $|U|=0$, we have that $|U+\lambda_k|=2$ and $|U-\lambda_k-1|=-3$. Thus $\upi_{U+\lambda_k}S(\lambda_k)_+\wedge H\M=\upi_{U-\lambda_k-1}S(\lambda_k)_+\wedge H\M=0$. The homology long exact sequence together with Lemma \ref{lem homology Slambda} give us the four exact sequences.

    Now we show that groups other than $H\M_U^G$ in the exact sequences can be reduced to $C_2$- and $G/C_2$-equivariant computations. Since the proofs for orientable and non-orientable $U$ are completely analogous, we assume $U$ is orientable. 
    
    In part $(2)$, if $a_k=-1$, then $U+\lambda_k$ is a $G/C_2$-representation and we know the first and last Mackey functors via $G/C_2$-computations. If $a_k\leq -2$, let $W=U-a_k\lambda_k$. By Proposition \ref{prop alambda iso}, $H\M_W\cong H\M_{U+\lambda_k}$ and the former Mackey functor can be computed $G/C_2$-equivariantly. Similarly for $H\M_{W-1}\cong H\M_{U+\lambda_k-1}$. Then we get the structures of the relevant homotopy Mackey functors above the group $C_2$. The rest of the structure can be deduced from $H\left(\M\downarrow^G_{C_2}\right)_{\star}$.

    In part $(3)$, if $a_k=1$, the computations of the first and last Mackey functors can be reduced to computations over $G/C_2$. If $a_k\geq 2$, let $W=U-a_k\lambda_k$. We have $H\M_{U-\lambda_k}\cong H\M_{W}$ and $H\M_{U-\lambda_k+1}\cong H\M_{W+1}$, thus the computations above the group $C_2$ can be reduced to $G/C_2$. $H\left(\M\downarrow^G_{C_2}\right)_{\star}$ tells us the rest of the structure.
\end{proof}

For the cases $|U|=\pm 1$, we have the following result:
\begin{theorem}[Induction Theorem B]
\label{thm induction 2}
    Let $U=a_0+a_1\alpha+\Sigma^{k}_{i=2}a_i\lambda_i\in RO(G)$. Then we have the following:
    
    When $|U|=1$,
    \begin{enumerate}
        \item If $a_k\leq 0$, then $H\M^G_U$ can be computed from $H\left(\M\downarrow^G_{C_2}\right)_{\star}$ and $H\left(\M^{\sharp C_2}\right)_{\star}$.
        \item If $a_k\geq 1$ and $U$ is orientable, then we have the exact sequence of Mackey functors
        \[
        FP(\M(e))\to H\M_U \xrightarrow{a_{\lambda}} H\M_{U-\lambda_k}\to FQ(\M(e))
        \]
        where the third Mackey functor can be computed from $H\left(\M^{\sharp C_2}\right)_\star$ and $H\left(\M\downarrow^G_{C_2}\right)_{\star}$.
        \item If $a_k\geq 1$ and $U$ is non-orientable, then we have the exact sequence of Mackey functors
        \[
        TFP(\M(e))\to H\M_U \xrightarrow{a_{\lambda}} H\M_{U-\lambda_k}\to TFQ(\M(e))
        \]
        where the third Mackey functor can be computed from $H\left(\M^{\sharp C_2}\right)_{\star}$ and $H\left(\M\downarrow^G_{C_2}\right)_{\star}$.
    \end{enumerate}
    
    When $|U|=-1$,
    \begin{enumerate}
        \item If $a_k\geq 0$, then $\upi_U H\M$ can be computed from $H\left(\M^{\sharp C_2}\right)_{\star}$ and $H\left(\M\downarrow^G_{C_2}\right)_{\star}$.
        \item If $a_k\leq -1$ and $U$ is orientable, then we have the exact sequence of Mackey functors
        \[
        FP(\M(e))\to H\M_{U+\lambda_k} \xrightarrow{a_{\lambda}} H\M_{U}\to FQ(\M(e))
        \]
        where the second Mackey functor can be computed from $H\left(\M^{\sharp C_2}\right)_{\star}$ and $H\left(\M\downarrow^G_{C_2}\right)_{\star}$.
        \item If $a_k\leq -1$ and $U$ is non-orientable, then we have the exact sequence of Mackey functors
        \[
        TFP(\M(e))\to H\M_{U+\lambda_k}\xrightarrow{a_{\lambda}} H\M_{U}\to TFQ(\M(e))
        \]
        where the second Mackey functor can be computed from $H\left(\M^{\sharp C_2}\right)_{\star}$ and $H\left(\M\downarrow^G_{C_2}\right)_{\star}$.
    \end{enumerate}
\end{theorem}
\begin{proof}
    The proof for the cases $|U|=\pm 1$ are analogous, and we only prove it for $|U|=1$. The case $a_k=0$ follows from Corollary \ref{reduction to quotient group}. The exact sequences comes from the cofiber sequence $S(\lambda)_+\to S^0\to S^{\lambda}$.

    If $a_k\leq -1$, let $W=U-a_k\lambda_k$. By Proposition \ref{prop alambda iso}, we have the isomorphism $H\M_{U}\cong H\M_W$, and the latter Mackey functor can be reduced to the homotopy of $H\M^{\sharp C_2}$. Then we understand the Mackey functor structure above $C_2$. $H\left(\M\downarrow^G_{C_2}\right)_{\star}$ tells us the rest.

    If $a_k\geq 1$, let $W=U-a_k\lambda_k$. By Proposition \ref{prop alambda iso}, $H\M_{U-\lambda_k}\cong H\M_{W}$ and the latter Mackey functor can be reduced to the homotopy of $H\M^{\sharp C_2}$. The rest of the structure can be deduced from $H\left(\M\downarrow^G_{C_2}\right)_{\star}$.
\end{proof}

\section{The positive cone of \texorpdfstring{H\M}{} for arbitrary \texorpdfstring{$\M$}{}}
In this section we will show how to use the methods described in previous sections to describe the positive cone of $H\M^G_{\star}$ for an arbitrary Mackey functor $\M$. Here the positive cone $H\M^G_{pos}$ is the subgroup of $H\M^G_\star$ indexed by $V\in RO(G)$ such that $a_i\leq 0$ for $1\leq i\leq k$. As examples, we completely compute the positive cone of $H\Z$ and $H\A$.

\subsection{Structure of the positive cone for arbitrary \texorpdfstring{$\M$}{}}

Note that in the positive cone, all indexing representations will have $|V|\leq 0$, since a $n$-dimensional $G$-CW complex cannot have a non-zero $(n+1)$-th homology. By Proposition \ref{prop alambda iso} and \ref{prop action of alambda}, multiplication by $a_{\lambda_k}$ is an isomorphism when $|V|\leq -1$, and it is projection to $\coker(\tr_e^G)$ when $|V|=0$. So the $V$-th homotopy groups with $|V|=0,-1$ are the only groups we need to describe. Note that when $V$ is orientable with $|V|=0$, this is discussed in Proposition \ref{prop mackey uv}.

\subsubsection{$H\M^G_V$ for non-orientable $V$ with $|V|=0$}
\label{subsec non-orientable V tot deg 0}
We firstly describe $H\M^G_V$ for non-orientable $V$ with $|V|=0$.

\begin{proposition}
\label{prop 1-alpha}
We have that
\[
H\M^G_{1-\alpha}=\ker(\tr_{G'}^G)
\]
and
\[
H\M^G_{n-n\alpha}=\prescript{}{\zeta_1}\M(G')
\]
for odd $n>1$, where $\zeta_1=1+\gamma$ (see Definition \ref{def zeta elements}).
\end{proposition}
\begin{proof}
Directly follows from Theorem \ref{keythm}.
\end{proof}
\begin{notation}
Let $V=\sum_{1\leq i\leq k}a_i\lambda_i$ be an actual orientable representation (recall that $\lambda_1=2\alpha$). Recall also that elements $u_W$, $u_{W'}$ commute for irreducible (thus general) orientable $W$ and $W'$ (see Proposition \ref{prop uv commute}). We will write \[u_V:=u^{c_1}_{\lambda_1}u_{\lambda_2}^{c_2}\ldots u_{\lambda_k}^{c_k}\in H\A^G_{|V|-V}.\]
Note that $u_V\colon H\M^G_0\to H\M^G_{|V|-V}\cong \M(G_V)^G$ is the restriction $\res^G_{G_V}$.
\end{notation}
\begin{proposition}
\label{prop mult by u on 1-alpha}
Let $V=1-\alpha+|V'|-V'$ be a non-orientable representation with $V'=\sum_{1\leq i\leq k}a_i\lambda_i\neq 0$. Then:
\begin{enumerate}
\item $H\M^G_{V}=\prescript{}{\zeta_1}\M(G_{V'})$.
\item Multiplication $u_{V'}\colon H\M^G_{1-\alpha}\cong \ker(\tr_{G'}^G)\to H\M^G_{V}\cong\prescript{}{\zeta_1}\M(G_{V'})$ is the restriction map $\res^{G'}_{G_V}$.
\end{enumerate}
\end{proposition}
Note that in Point (2) image of $\res^{G'}_{G_V}$ always lies in the codomain when restricted to the domain, by the double-coset formula.
\begin{proof}
\leavevmode
\begin{enumerate}
\item Follows from Theorem \ref{keythm}.
\item Let $v\colon S^1\to S^\alpha\wedge H\M$ be a representative of an element in $H\M^G_{1-\alpha}$. Then consider the commutative diagram
\[
\begin{tikzcd}
S^{|V'|}\ar[r,"v\wedge\id"]& S^{\alpha-1}\wedge S^{|V'|}\wedge H\M\ar[r,"\id\wedge u_{V'}"]\ar[dr]&S^{\alpha-1}\wedge S^{V'}\wedge H\M\ar[d] \\
&&S^{\alpha-1}\wedge{G/G_V'}_+\wedge S^{|V'|}\wedge H\M.
\end{tikzcd}
\]

The diagonal map displays the restriction of the class in $H\M^G_{1-\alpha}$ to $H\M^{G_{V'}}_{1-\alpha}\cong H\M^{G_{V'}}_0=\M(G_{V'})$, and the vertical map is the inclusion of $\prescript{}{\zeta_1}\M(G_{V'})$ in $\M(G_{V'})$ when taking $\pi_{|V'|}^G(-)$.

The restriction map $\res(1-\alpha)^G_{G_V'}$ is induced by the restriction in the Mackey functor $\M$. This can be seen as follows. Let $v\colon S^1\to S^\alpha\wedge H\M$ be a representative of a class in $H\M^G_{1-\alpha}$. Then its restriction to $H\M^{G_{V'}}_{1-\alpha}$ is given by the composite
\[
G/{G_{V'}}_+\wedge S^{1-\alpha}\to S^{1-\alpha}\to H\M.
\]
Note that $G/{G_{V'}}_+\wedge S^{1-\alpha}\cong G/{G_{V'}}_+$ by the shearing isomorphism. Thus restriction of $v$ gives a class in $\M(G_{V'})$, as expected.

Therefore since the diagonal map is induced by $\res^G_{G_{V'}}$ and the vertical map is an inclusion, the multiplication by $u_V$ has to be the restriction.
\end{enumerate}
\end{proof}
\subsubsection{$H\M^G_V$ for $V$ with $|V|=-1$}
\label{subsubsec v tot deg -1}
Now we will compute the entries in the positive cone graded over $V$ with $|V|=-1$. We first observe that if $V$ has $|V|=-1$, then there exists a unique orientable $G$-representation $U$ with $U^G=0$ such that $V=|U|-U-1$ or $|U|-U-\alpha$ depending on the orientability of $V$.

\begin{proposition}
\label{prop HM of deg -1}
    Let $U=\sum_{i=1}^j a_i\lambda_i, a_i\geq 0$ and $a_j\neq 0$. Let $W=U-\lambda_j$ and $s$ be the biggest integer such that $\lambda_s$ appears in $W$ with non-zero coefficient. Then
    \begin{equation*}
        H\M_{|U|-1-U}^G=\begin{cases}
             \frac{_{\zeta_j}\M(G_U)}{(1-\gamma)\M(G_U)}&\,\text{if}\,a_j\geq 2,\\[7pt]
             \frac{\ker(\zeta_s\tr^{G_{W}}_{G_{U}})}{(1-\gamma)\M(G_U)}&\,\text{if}\,a_j=1\,\text{and}\,W\neq 0,\\[7pt]
             \frac{\ker(\tr^{G}_{G_{U}})}{(1-\gamma)\M(G_U)}&\,\text{if}\,a_j=1\,\text{and}\,W=0.
        \end{cases}
    \end{equation*}
    and
    \begin{equation*}
        H\M_{|U|-\alpha-U}^G=\left\{
        \begin{array}{ll}
             \frac{_{\overline{\zeta}_j}\M(G_U)}{(1+\gamma)\M(G_U)}&\,\text{if}\,a_j\geq 2,\\[7pt]
             \frac{\ker(\overline{\zeta}_s\tr^{G_{W}}_{G_{U}})}{(1+\gamma)\M(G_U)}&\,\text{if}\,a_j=1\,\text{and}\,W\neq0,\\[7pt]
             \frac{\ker(\overline{\zeta}_1\tr^{G'}_{G_{U}})}{(1+\gamma)\M(G_U)}&\,\text{if}\,a_j=1\,\text{and}\,W=0.
        \end{array}
        \right.
    \end{equation*}
\end{proposition}
\begin{proof}
    Directly follows from Theorem \ref{keythm}.
\end{proof}

In the next two subsections, we will use the theory described above to derive the positive cones of the most important Eilenberg-MacLane spectra $H\Z$ and $H\A$ over $C_{2^k}$. Note the whole $RO(G)$-graded homotopy of $H\underline{\mathbb{F}_2}$, in particular the positive cone, is determined by the second author \cite{Yan23HF2}.

\subsection{The constant Mackey functor \texorpdfstring{$\Z$}{}}
In this subsection we will describe the structure of the positive cone of $H\Z$. Recall the Mackey functor $\Z$ is given by $\Z(H)=\mathbb{Z}$ for $H\subset G$. Using transitivity, the restrictions and transfers are determined by $\res^H_K=1$ and $\tr^H_K=2$ when $[H:K]=2$.  Since $\Z$ is actually a Green functor, the positive cone of $H\Z$ has a ring structure. The reader can compare the following result with \cite{HHRc}.

\begin{theorem}
\label{thm pos cone HZ}
Let $G=C_{2^k},k\geq 1$. Then we have
\[
H\Z^G_{pos}=\frac{\mathbb{Z}\left[a_\alpha,\{a_{\lambda_i}|2\leq i\leq k\},\{u_{\lambda_i}|1\leq i\leq k\}\right]}{2a_\alpha,\{2^ia_{\lambda_i}|2\leq i\leq k\}, \{2^{i-j}a_{\lambda_i}u_{\lambda_j}-a_{\lambda_j}u_{\lambda_i}|1\leq j < i\leq k\}}.
\]
\end{theorem}
\begin{proof}
The result is true for $k=1$ by \cite{Duggerthesis}, \cite{Gre18} and \cite{Zeng17}. It is true for $k=2$ by \cite{Yan22HZ}. Assume it is true for $k-1$. For $G=C_{2^{k}}$, the relations can be readily deduced from Proposition \ref{prop mackey av} and Proposition \ref{prop au relation}. We are left to show that we have the additive splitting
\begin{equation}
\label{splitting of HZ}
    \begin{aligned}
        H\Z_{pos}^G&=\mathbb{Z}[a_{\lambda_k}][a_{\alpha},\cdots,a_{\lambda_{k-1}},u_{2\alpha},\cdots,u_{\lambda_{k-1}}]/\{2^k a_{\lambda_k},\text{relations for}\,C_{2^{k-1}}\}\\
        &\oplus\mathbb{Z}\langle u_{\lambda_k}^i \rangle_{i\geq 1}[a_{\lambda_k}][\uta,\cdots,u_{\lambda_{k-1}}]\langle 1,\aal\rangle/\{2\aal,2^k a_{\lambda_k}\}
    \end{aligned}
\end{equation}
For the notations, see the introduction (Section \ref{subsec notations}). Let $V$ be in the positive cone. Then we get non-trivial classes only when $|V|\leq 0$, since a $n$-dimensional $G$-CW complex cannot have $(n+1)$-dimensional homology. By Proposition \ref{reduction to quotient group} and the induction assumption, we get the classes
\begin{equation}
\label{pos cone C2k-1}
\mathbb{Z}[a_{\alpha},\cdots,a_{\lambda_{k-1}},u_{2\alpha},\cdots,u_{\lambda_{k-1}}]/\{\text{relations for}\,C_{2^{k-1}}\}
\end{equation}
when $V$ does not contain any copies of $\lambda_k$. Write $V=V^{\sharp}+a_k\lambda_k$ where $V^{\sharp}$ does not contain $\lambda_k$. Note $a_k\geq 0$ since $V$ is in the positive cone. We use $H\Z^G_{{pos}^{\sharp}}$ to denote the subring of the first summand in \eqref{splitting of HZ} with $a_{\lambda_k}$ having $0$-th power. That is, ${pos}^{\sharp}$ is the direct sum of all representations in the positive cone without any copy of ${\lambda_k}$. 

If $|V^{\sharp}|\leq 0$, then $H\Z_{{pos}^{\sharp}}^G$ is known as in \eqref{pos cone C2k-1}, and Propositions \ref{prop action of alambda} and \ref{prop alambda iso} gives us the first direct summand of \eqref{splitting of HZ}, since multiplication by $a_{\lambda_k}$ will either be an isomorphism or a surjection. If $|V^{\sharp}|> 0$, let $c$ be the unique integer such that $|V-c\lambda_k|=0$ or $-1$ depending on whether $|V|$ is even or odd. Let $d=a_k-c$, then we have $V=V^{\sharp}-c\lambda_k-d\lambda_k$. Write $W=V^{\sharp}-c\lambda_k$. We will show that the groups $H\Z_{V}^G$ with $|V^{\sharp}|>0$ consist of precisely the second summand in \eqref{splitting of HZ}.

By Proposition \ref{prop mackey uv} and an easy generalization of Proposition \ref{prop uv commute}, we know the $u_{\lambda_i}$'s generate a subring $\mathbb{Z}[\{u_{\lambda_i}|1\leq i\leq k\}]$ which gives us the groups $H\Z_{W}^G=0$ with $|W|=0$, since $H\Z_{n-n\alpha}=0$ for odd $n$ by Proposition \ref{prop 1-alpha}. Multiplying by powers of $a_{\lambda_k}$ gives us all the classes in the gradings with $|V^{\sharp}|>0$ and $|W|=0$ by Propositions \ref{prop action of alambda} and \ref{prop alambda iso}.

By Proposition \ref{prop HM of deg -1}, for $U$ an actual $G$-representation, we have that $H\Z^G_{|U|-U-1}=0$ and $H\Z^G_{|U|-U-\alpha}=\mathbb{Z}/2$.

We look at the following exaxt sequence
\[\begin{tikzcd}
H\Z^{G'}_{|U|-U}\rar["\tr_{G'}^G"]&H\Z^{G}_{|U|-U}\rar["a_\alpha"]&
H\Z^{G}_{|U|-U-\alpha}\rar&H\Z^{G'}_{|U|-U-1}.
\end{tikzcd}
\]
Since $\tr^G_{G'}=2$, we see $\aal$ is the projection $\mathbb{Z}\to \mathbb{Z}/2$, and we obtain the classes $\mathbb{Z}/2[\{u_{\lambda_i}|1\leq i\leq k\}]\langle\aal \rangle$. Multiplying by powers of $a_{\lambda_k}$ gives us all the classes in gradings with $|V^{\sharp}|>0$ and $|W|=-1$ by Propositions \ref{prop action of alambda} and \ref{prop alambda iso}.

\end{proof}

\subsection{Burnside Mackey functor \texorpdfstring{\A}{}}
In this subsection we will describe the structure of the positive cone for $H\A$, where $\A$ is the Burnside Mackey functor (actually a Green functor). Recall that for a general finite group $G$, the Burnside Green functor is given by for $H\subset G$,
\[
\A(H)=\text{the Grothendieck ring on isomorphism classes of finite $H$-sets},
\]
where the addition is disjoint union and product is the Cartesian product. In particular, it has a free $\mathbb{Z}$-basis given by the isomorphism classes of $H/L,L\subset H$. We will use $[H/L]$ to denote the isomorphism class of $H/L$. For $K\subset H\subset G$, $\res^H_K$ is given by restricting the $H$-action to a $K$-action, and $\tr^H_K$ is given by $H\times_K(-)$.

Fix $G=C_{2^k}$ and consider $\A(C_{2^k})$. The $\mathbb{Z}$-basis will be denoted by $\omegai{k}_j$, where $\omegai{k}_j$ is the class corresponding to the isomorphism class $[C_{2^k}/{C_{2^{k-j}}}]$. In particular, $\omegai{k}_k$ corresponds to the free $G$-set and $\omegai{k}_0$ is the unity of the ring. Similarly we define $\omegai{k}_j$ for subgroups of $G$.

The restrictions are explicitly given by:
\begin{equation*}
    \res^{C_{2^i}}_{C_{2^{i-1}}}(\omegai{i}_j)=\left\{
    \begin{array}{ll}
    2\omegai{i-1}_{j-1}&\,\text{if}\,1\leq j\leq i,\\
    \omegai{i-1}_{0}&\,\text{if}\, j=0.
    \end{array}
    \right.
\end{equation*}
The transfers are given by
\[
\tr^{C_{2^i}}_{C_{2^{i-1}}}(\omegai{i-1}_j)=\omegai{i}_{j+1}
\]
for $0\leq j\leq i-1$.

By Proposition \ref{prop mackey uv} and the fact that Weyl actions in $\A$ are trivial, we have
\[
H\A_{2-\lambda_j}(H)=\left\{\begin{array}{ll}
     \A(C_{2^{k-j}})&\,\text{if}\,C_{2^{k-j}}\subset H,\\
     \A(H)&\,\text{otherwise.}
\end{array}
\right.
\]
It is the constant Mackey functor with value $\A(C_{2^{k-j}})$ from the $G/G$-level to the $G/{C_{2^{k-j}}}$-level. At the $G/{C_{2^{k-j}}}$-level and below, it is the Burnside Mackey functor for the group $C_{2^{k-j}}$.

At the $C_{2^{k-j}}$-level, we have the actual products $\omegai{k-j}_{i}\cdot u_{\lambda_j}$ for $0\leq i\leq k-j$ and they form a $\mathbb{Z}$-basis of $H\A_{2-\lambda_j}(G/C_{2^{k-j}})$. Since $\res^G_{C_{2^{k-j}}}=1$ in this degree, we denote the lifts of these classes to the $G/G$-level by $[\omegai{k-j}_{i} u_{\lambda_j}]$, since they are not actual products anymore.

For a general orientable $V=\Sigma^j_{i=1}a_i\lambda_i$ with $a_j\neq 0$, the group $H\A_{|V|-V}^G$ has a $\mathbb{Z}$-basis $[\omegai{k-j}_{i} u_V],0\leq i\leq k-j$.

\begin{theorem}
\label{thm pos cone HA}
Let $G=C_{2^k}$. We have
\[
H\A^{G}_{pos}=\A(G)\left[a_\alpha,
\{a_{\lambda_i}|2\leq i\leq k\},\{[\omegai{k-j}_{i}u_{\lambda_j}]|1\leq j\leq k, 0\leq i\leq k-j\}
\middle]\right/ S,
\]
where $S$ is the set of relations consisting of the following:
\begin{enumerate}
\item For all $1\leq j\leq k$ we have that $\omegai{k}_ja_\alpha=0$.
\item For all $2\leq i\leq j\leq k$ we have that $\omegai{k}_ja_{\lambda_i}=0$.
\item For all $1\leq j_1\leq j_2\leq k$ and $0\leq i_1\leq k-j_1,0\leq i_2\leq k-j_2$, we have that
\begin{equation*}
    [\omegai{k-j_1}_{i_1}u_{\lambda_{j_1}}]\cdot [\omegai{k-j_2}_{i_2}u_{\lambda_{j_2}}]=\left\{
    \begin{array}{ll}
2^{i_1}[\omegai{k-j_2}_{i_2}u_{\lambda_{j_2}}]u_{\lambda_{j_1}}&\,\text{if}\,i_1+j_1\leq j_2,\\
2^{j_2-j_1+i_1}[\omegai{k-j_2}_{i_2}u_{\lambda_{j_2}}]u_{\lambda_{j_1}}&\,\text{if}\,j_2\leq i_1+j_1\leq i_2+j_2,\\
2^{j_2-j_1+i_2}[\omegai{k-j_2}_{i_1+j_1-j_2}u_{\lambda_{j_2}}]u_{\lambda_{j_1}}&\,\text{if}\,i_1+j_1> i_2+j_2.
    \end{array}
    \right.
\end{equation*}
In particular, we have that for all $1\leq l\leq j\leq k$,
\[
u_{\lambda_l}\cdot [\omegai{k-j}_i u_{\lambda_j}]=[\omegai{k-j}_i u_{\lambda_l}u_{\lambda_j}].
\]
\item For all $1\leq i,j\leq k$, we have that
\begin{equation*}
    \omegai{k}_{i_1}\cdot [\omegai{k-j}_{i_2}u_{\lambda_i}]=\left\{
    \begin{array}{ll}
         2^{i_1+i_2}[\omegai{k-j}_{k-j-i_1}u_{\lambda_i}]&\,\text{if}\,i_1\leq j,\,\text{and}\,i_1\leq k-j-i_2,\\
         2^{k-j} [\omegai{k-j}_{i_2}u_{\lambda_i}]&\,\text{if}\,i_1\leq j,\,\text{and}\,i_1> k-j-i_2,\\
         2^{j+i_2} [\omegai{k-j}_{k-j-i_1}u_{\lambda_i}]&\,\text{if}\,i_1> j,\,\text{and}\,i_1\leq k-j-i_2,\\
         2^{k-i_1} [\omegai{k-j}_{i_2}u_{\lambda_i}]&\,\text{if}\,i_1> j,\,\text{and}\,i_1> k-j-i_2.
    \end{array}
    \right.
\end{equation*}
\item For all $1\leq j<i\leq k$ we have that $a_{\lambda_i}[\omegai{k-j}_{i-j}u_{\lambda_j}]=a_{\lambda_j}u_{\lambda_i}$.
\end{enumerate}
\end{theorem}
\begin{proof}
The case when $k=1$ is shown in \cite{Sikora22}. Assume the result is true for $k-1$. We use the same notations $V=V^{\sharp}-a_k\lambda_k$ and $W=V^{\sharp}-c\lambda_k$ when $|V^{\sharp}|>0$ as in Theorem \ref{thm pos cone HZ}. 

We first deduce the relations. Since $\A$ is a Green functor, we have the Frobenius relation: for any $a\in H\A_V(K)$ and $b\in H\A_W(H)$ 
\[
\tr^H_K(a)\cdot b=\tr^H_K(a\cdot \res^H_K(b)).
\]
The relations $(1)$ and $(2)$ are direct consequences of the Frobenius relation. Relation $(5)$ follows from the gold relation in Proposition \ref{prop au relation}. More precisely, we have $G_{\lambda_i}=C_{2^{k-i}}, G_{\lambda_j}=C_{2^{k-j}}$ and $\tr^{C_{2^{k-j}}}_{C_{2^{k-i}}}(1)=\omegai{k-j}_{i-j}$.

For $(3)$, we look at the following diagram coming from the pairing $H\A_{2-\lambda_{j_1}}\square H\A_{2-\lambda_{j_2}}\xrightarrow{\mu} H\A_{4-\lambda_{j_1}-\lambda_{j_2}}$:
\[\xymatrix{
\A(C_{2^{k-j_1}})\otimes \A(C_{2^{k-j_2}})\ar[r]^-{\mu}\ar[d]_{1\otimes 1}&\A(C_{2^{k-j_2}})\ar[d]_1\\
\A(C_{2^{k-j_1}})\otimes \A(C_{2^{k-j_2}})\ar[r]^-{\mu}\ar[d]_{{\res^{C_{2^{k-j_1}}}_{C_{2^{k-j_2}}}}\otimes 1}&\A(C_{2^{k-j_2}})\ar[d]_1\\
\A(C_{2^{k-j_2}})\otimes \A(C_{2^{k-j_2}})\ar[r]^-{\mu}&\A(C_{2^{k-j_2}}).
}
\]
Since the target has restrictions given by $1$ above $G/{C_{2^{k-j_2}}}$-level by Propostion \ref{prop mackey uv}, to understand the product on the $G/G$-level, it suffices to restrict it to the $G/{C_{2^{k-j_2}}}$-level. We notice
\[
{\res^{C_{2^{k-j_1}}}_{C_{2^{k-j_2}}}}(\omegai{k-j_1}_{i_1})=
\begin{cases}
     2^{i_1}&\text{if}\,i_1+j_1\leq j_2,\\
     2^{j_2-j_1}\omegai{k-j_2}_{i_1+j_1-j_2}&\text{if}\, i_1+j_1>j_2.
\end{cases}.
\]
Together with
\[
\omegai{k-j_2}_{i_1+j_1-j_2}\cdot \omegai{k-j_2}_{i_2}=
\begin{cases}
     2^{i_2}\omegai{k-j_2}_{i_1+j_1-j_2}&\text{if}\,i_1+j_1>i_2+j_2,\\
     2^{i_1}\omegai{k-j_2}_{i_2}&\text{if}\, i_1+j_1\leq i_2+j_2,
\end{cases}
\]
we deduce that
\begin{equation*}
    \res^G_{C_{2^{k-i_2}}}([\omegai{k-j_1}_{i_1}u_{\lambda_{j_1}}]\cdot [\omegai{k-j_2}_{i_2}u_{\lambda_{j_2}}])=
    \begin{cases}
2^{i_1}\omegai{k-j_2}_{i_2}u_{\lambda_{j_1}}u_{\lambda_{j_2}}&\text{if}\,i_1+j_1\leq j_2,\\
2^{j_2-j_1+i_1}\omegai{k-j_2}_{i_2}u_{\lambda_{j_1}}u_{\lambda_{j_2}}&\text{if}\,j_2\leq i_1+j_1\leq i_2+j_2,\\
2^{j_2-j_1+i_2}\omegai{k-j_2}_{i_1+j_1-j_2}u_{\lambda_{j_1}}u_{\lambda_{j_2}}&\text{if}\,i_1+j_1> i_2+j_2.
    \end{cases}
\end{equation*}
Since $\res^G_{C_{2^{k-i_2}}}=1$ in $H\A_{4-\lambda_1-\lambda_2}$, relation $(3)$ follows. Note that since $G_{\lambda_{j_2}}\subset G_{\lambda_{j_1}}$, $[\omegai{k-j_2}_{i_2}u_{\lambda_{j_1}}u_{\lambda_{j_2}}]=[\omegai{k-j_2}_{i_2}u_{\lambda_{j_2}}]u_{\lambda_{j_1}}$ by Proposition \ref{prop uv commute}. For relation $(4)$, we use the pairing $H\A_{0}\square H\A_{2-\lambda_i}\xrightarrow{\mu}H\A_{2-\lambda_i}$ and the argument is similar to the that of $(3)$.

Using the gold relation (see Proposition \ref{prop au relation}), to prove the structure of $H\A_{pos}^G$, we are left to show that we have the additive splitting
\begin{equation}
\label{splitting of HA}
H\A_{pos}^G=
\begin{aligned}
        &\frac{\A(G)[a_{\lambda_k}]\left[a_{\alpha},\{a_{\lambda_i}|2\leq i\leq k-1\},\left\{[\omegai{k-j}_{i}u_{\lambda_j}]\;\middle|\;1\leq j\leq k-1, 0\leq i\leq k-j\right\}\right]}{\widetilde{S}}\\
        &\oplus\frac{\A(G)\langle u_{\lambda_k}^i \rangle_{i\geq 1}[a_{\lambda_k}]\left[\left\{[\omegai{k-j}_{i}u_{\lambda_j}]\;\middle|\;1\leq j\leq k-1, 0\leq i\leq k-j\right\}\right]\langle 1,\aal\rangle}{2\aal}.
\end{aligned}
\end{equation}
The relations in $\widetilde{S}$ consists of the relations $(1),(2)$, as well as the relations $(3),(4),(5)$ where only $u_i$'s with $i\neq k$ are involved. The proof of the second summand in \eqref{splitting of HA}, which corresponds to the gradings with $|V^{\sharp}|>0$, is completely analogous to the case of $H\Z$. We only need to consider the first summand in \eqref{splitting of HA}, which correspond to $|V^{\sharp}|\leq 0$.

Note that
\begin{equation}
\label{splitting of A}
    (\A_G)^{\sharp C_2}\cong \A_{G'}\oplus \Z^{*}_{G'}
\end{equation}
where $\M_{H}$ is to indicate that $\M$ is the relevant $H$-equivariant Mackey functor. $\Z^*$ is the level-wise dual of $\Z$ generated by the free sets, ie, $\Z^*(C_{2^i})=\mathbb{Z}\langle\omegai{i}_{0}\rangle$ and $\res^H_K=2$ and $\tr^H_K=1$ when $[H:K]=2$. Under the splitting, $\omegai{k-j}_i\in \A(G)$ with $0\leq i<k-j$ on the left correspond to $\omegai{k-j-1}_i\in \A(G')$ on the right.

Thus we have the splitting of $G'$-equivariant spectra
\begin{equation*}
    (H\A_G)^{C_2}\simeq H\A_{G'}\vee H\Z^{*}_{G'}.
\end{equation*}
We consider the subring $H\A^G_{{pos}^{\sharp}}$ of the first summand in \eqref{splitting of HA}, where ${pos}^{\sharp}$ are all representations of the positive cone without any copy of $a_{\lambda_k}$. It corresponds to the classes with $a_{\lambda_k}$ having $0$-th power and additively decomposes as
\begin{equation}
\label{splitting of pos cone HA}
H\A^G_{{pos}^{\sharp}}\cong(H\A_{G'})_{pos}^{G'}\oplus (H\Z^*_{G'})_{pos}^{G'}    
\end{equation}
via Corollary \ref{reduction to quotient group},
with
\[
(H\A_{G'})_{pos}^{G'}=\frac{\A(G')\left[a_\alpha,
\{a_{\lambda_i}|2\leq i\leq k-1\},\{[\omegai{k-1-j}_{i}u_{\lambda_j}]|1\leq j\leq k-1, 0\leq i\leq k-j\}
\right]}{S'}
\]
by induction. Here $S'$ is the set of relations for the group $G'$. We notice that, via the identification \eqref{splitting of pos cone HA}, the difference between this ring and $H\A^G_{pos^{\sharp}}$ consists of the groups
\[
\mathbb{Z}\left\langle\{[\omegai{k-j}_{k-j}u_{U}]|U=\Sigma_{i=1}^{j}a_i\lambda_j \text{ with } a_j\neq 0, 1\leq j\leq k-1\}
\right\rangle
\]
since $\omegai{k}_k$ kills any Euler class by the Frobenius relation. These groups comes from $(H\Z^*_{G'})_{pos}^{G'}$ via Proposition \ref{prop mackey uv}. Taking into account of the particular case of relation $(3)$, we get a surjective ring map
\[
\A(G)\left[a_\alpha,
\{a_{\lambda_i}|2\leq i\leq k\},\{[\omegai{k-j}_{i}u_{\lambda_j}]|1\leq j\leq k, 0\leq i\leq k-j\}
\right]\twoheadrightarrow H\A^{G}_{pos},
\]
and the relation set $S$ is proved above.

\begin{remark}
    As mentioned in Section \ref{sec rog grading}, this result essentially contains the odd primary result since we are working locally. Let $p$ be an odd prime, and $G=C_{p^n}$. Recall the convention that $\lambda_1=2\alpha$, $p$-locally, the positive cone in this case would be
    \[
    H\A^{G}_{pos}=\A(G)\left[
    \{a_{\lambda_i}|1\leq i\leq k\},\{[\omegai{k-j}_{i}u_{\lambda_j}]|1\leq j\leq k, 0\leq i\leq k-j\}
    \middle]\right/ S,
    \]
    In the definition of $\omegai{k}_j$ and in the relations $(1)$-$(5)$, $2$ should be changed to $p$.
\end{remark}

\end{proof}

\section{The negative cone of \texorpdfstring{$H\M$}{} for arbitrary \texorpdfstring{$\M$}{}}
In this section we determine the structure of the negative cone $H\M^G_{neg}$ of $H\M^G_{\star}$, which is the subgroup indexed by $V\in RO(G)$ such that $a_i\geq 0$ for $1\leq i\leq k$. As an example, we completely compute the negative cone of $H\Z$. 

\subsection{Structure of the negative cone for arbitrary $\M$}
The homotopy groups with $|V|\leq -1$ in the negative cone are $0$ since $n$-dimensional CW-complex does not have $n+1$-dimentional cohomology. As a consequence, we only need to care about the $V$-th homotopy groups where $|V|=0,1$.
\subsubsection{\texorpdfstring{$H\M_V^G$}{} for orientable \texorpdfstring{$V$}{} with \texorpdfstring{$|V|=0$}{}}
\begin{proposition}
\label{prop u inverse}

    Let $V$ be an orientable (not necessarily irreducible) $G$-representation. Then $H\M^G_{V-|V|}\cong \M(G_V)/G$. For the Mackey functor structure, we have
\[
H\M_{V-|V|}(H)=
\begin{cases}
\M(G_V)/H&\textrm{if }G_V\subset H,\\
\M(H)&\textrm{otherwise.}
\end{cases}
\]
For $K\subset H\subset G_V$, the maps $\res^H_K,\tr^H_K$ are induced from that of $\M$. If $G_V\subset K\subset H$, then $\tr^H_K$ is induced by the canonical projection, and $\res^H_K$ is given by $\res^{C_{2^i}}_{C_{2^{i-1}}}(x)=(1+\gamma^{2^{k-i}})\cdot x$ and transitivity. Moreover, the map $u_V:\M(G_V)/G\cong H\M^G_{V-|V|}\to H\M_0^G=\M(G)$ is induced by $\tr^G_{G_V}$.
\end{proposition}
\begin{proof}
    Dual to the proof of Proposition \ref{prop mackey uv}.
\end{proof}

\subsubsection{\texorpdfstring{$H\M_V^G$}{} for non-orientable $V$ with $|V|=0$}
\begin{proposition}
\leavevmode
\begin{enumerate}
    \item The Mackey functor structure of $H\M^G_{\alpha-1}$ is given by
    \[
H\M_{\alpha-1}(H)=
\begin{cases}
\M(G')/{\im(\res^G_{G'})}&\textrm{if }H=G,\\[5pt]
\widetilde{\M(H)}&\textrm{otherwise.}
\end{cases}
\]
The map $\tr^G_{G'}$ is the canonical projection, and $\res^G_{G'}(x)=(1-\gamma)\cdot x$. Transfers and restrictions on other levels are induced from that of $\M$.
\item The Mackey functor structure of $H\M_{n\alpha-n}^G$ for odd $n>1$ is given by
\[
H\M_{n\alpha-n}(H)=
\begin{cases}
\M(G')/{\zeta_1\cdot \M(G')}&\textrm{if }H=G,\\[5pt]
\widetilde{\M(H)}&\textrm{otherwise.}
\end{cases}
\]
The map $\tr^G_{G'}$ is the canonical projection, and $\res^G_{G'}(x)=(1-\gamma)\cdot x$. Transfers and restrictions on other levels are induced from that of $\M$.
\end{enumerate}
\end{proposition}
\begin{proof}
    We firstly note that since $\res^G_{G'}\alpha=1$, we have $H\M_{n\alpha-n}(H)\cong \M(H)$ for $H\subset G'$. Now for each $H\subset G'$, the Weyl action on $H\M^H_n(S^{n\alpha})$ comes from the action on $S^{n\alpha}$ (sign action for $n$ odd) as well as the action on $\M(H)$. We have that there is an isomorphism $H\M_{n\alpha-n}\downarrow^G_{G'}\cong {{\tilde{\ZZ}}}\otimes (H\M_0\downarrow^G_{G'})\cong {{\tilde{\ZZ}}}\otimes (\M\downarrow^G_{G'})$ (levelwise tensoring).
    
    For the Mackey functor structure between the levels $G$ and $G'$, our proof works for both $n=1$ and odd $n>1$, so we will assume $n=1$. Replacing the top differential by $\zeta_1$ in the cellular cochain complex from Theorem \ref{keythm} provides the proof for odd $n>1$. The sphere $S^{\alpha}$ fits into the cofiber sequence
    \[
    {G/G'}_+\xrightarrow{} {G/G}_+\to S^{\alpha}.
    \]
    Its cellular cochain complex with coefficient in $\M$ is
    \[
    0\to \M(G)\xrightarrow{\res^G_{G'}}\M(G')
    \]
    with differential induced from the folding map $\nabla: G/G'\to G/G$. To deduce the Mackey functor structure, we have to apply $G/G'\times -$ to this map and consider the following commutative diagram, where the vertical arrows induce the restrictions and transfers in the Mackey functor structures:
    \[
\xymatrix{
G/G\times G/G&G/G\times G/G'\ar[l]_{1\times \nabla} &\M(G)\ar[d]^{\res^G_{G'}}\ar[r]^{\res^G_{G'}}&\M(G')\ar[d]^{\Delta}\\
G/G'\times G/G\ar[u]_{\nabla\times 1}&G/G'\times G/G'\ar[l]_{1\times \nabla}\ar[u]_{\nabla\times 1},&\M(G')\ar[r]^-{(1,\gamma)}&\M(G')\oplus 
M(G').
}
\]
Here we made the identification
\[
G/G'\times G/G'\cong G/G'_{(e,e)}\coprod G/G'_{(\gamma, 1)},
\]
and the vertical maps correspond to the restriction $C^*_G(S^{\alpha};\M)\to C^*_{G'}(i^*_{G'}\underline{S^{\alpha}};\M)$. In the right square, the map $\Delta$ is the diagonal, and $(1,\gamma)(x)=(x,\gamma x)$. For any $x\in \M(G')$ from the top right corner,
\[
\Delta(x)=(x,x)=(x,0)+(0,x)=(0,-\gamma x)+(0,x)=(0,(1-\gamma)x)
\]
modulo the image of $(1,\gamma)$. From the construction of $\psi_1$ in the proof of Theorem \ref{keythm}, we deduce that in $H\M_{\alpha-1}$, we have $\res^G_{G'}(x)=(1-\gamma)\cdot x$. Note that the $G$-module structure on $H\M_{\alpha-1}^{G'}$ is actually $\widetilde{\ZZ}\otimes\M(G')$. By the double-coset formula, we deduce $\tr^G_{G'}=1$.
\end{proof}

\begin{proposition}
    We have that
    \[
    H\M^G_{\alpha-1}=\coker(\res^G_{G'})
    \]
    and
    \[
    H\M_{n\alpha-n}^G=\M(G')/{\zeta_1\cdot \M(G')}
    \]
    for odd $n>1$.
\end{proposition}
\begin{proof}
    Directly follows from Theorem \ref{keythm}.
\end{proof}

\begin{proposition}
\label{prop total dim 0 coho nonorient}
Let $V=\alpha-1+V'-|V'|$ be a non-orientable representation with $V'=\sum_{1\leq i\leq k}c_i\lambda_i\neq 0$, then:
\begin{enumerate}
\item $H\M^G_{V}=\frac{\M(G_V)}{\zeta_1\M(G_V)}$.
\item Multiplication $u_{V'}\colon H\M^G_V\cong \frac{\M(G_V)}{\zeta_1\M(G_V)}\to H\M^G_{\alpha-1}\cong\coker(\res^G_{G'})$ is the transfer map $\tr^{G'}_{G_V}$. 
\end{enumerate}
\end{proposition}
\begin{proof}
    Dual to Proposition \ref{prop mult by u on 1-alpha}.
\end{proof}

\subsubsection{$H\M^G_V$ for $V$ with $|V|=1$}

\begin{proposition}
\label{prop cohomology of total deg 1}
    Let $U=\sum_{i=1}^j a_i\lambda_i, a_i\geq 0$ and $a_j\neq 0$. Let $W=U-\lambda_j$ and $s$ be the biggest integer such that $\lambda_s$ appears in $W$ with non-zero coefficient. Then
    \begin{equation*}
        H\M_{U-|U|+1}^G=
        \begin{cases}
             \frac{\M(G_U)^G}{\prescript{}{\zeta_j}\M(G_U)}&\text{if}\,a_j\geq 2,\\[7pt]
             \frac{\M(G_U)^G}{\im({\zeta_s}\res^{G_W}_{G_U})}&\text{if}\,a_j=1\,\text{and}\,W\neq 0\\[7pt]
             \frac{\M(G_U)^G}{\im(\res^{G}_{G_U})}&\text{if}\,a_j=1\,\text{and}\,W=0.
        \end{cases}
    \end{equation*}
    and
    \begin{equation*}
        H\M_{U-|U|+\alpha}^G=
        \begin{cases}
             \frac{_{\zeta_1}\M(G_U)}{_{\overline{\zeta}_j}\M(G_U)}&\text{if}\,a_j\geq 2,\\[7pt]
             \frac{_{\zeta_1}\M(G_U)}{\im(\overline{\zeta}_s\res^{G_W}_{G_U})}&\text{if}\,a_j=1\,\text{and}\,W\neq 0,\\[7pt]
             \frac{_{\zeta_1}\M(G_U)}{\im(\overline{\zeta}_1\res^{G'}_{G_U})}&\text{if}\,a_j=1\,\text{and}\,W= 0.
        \end{cases}
    \end{equation*}
\end{proposition}
\begin{proof}
    Directly follows from Theorem \ref{keythm}.
\end{proof}

\subsection{The constant Mackey functor \texorpdfstring{$\Z$}{}}
In the following theorem we compute the negative cone of $H\Z$. The notations comes from the previous computations via Tate squares in \cite{Zeng17}, \cite{Yan22HZ} and \cite{Yan23HF2}. For the methods to deduce the entire Mackey functor structure, we refer the reader to \cite[Section 6]{Yan22HZ}. Briefly speaking, by induction, we are only concerned with $\tr^G_{G'}$ and $\res^G_{G'}$, which can be deduced from $\aal$ multiplications by a trick of Hill-Hopkins-Ravenel in \cite[Lemma 4.2]{HHRb}.

\begin{theorem}
\label{thm neg cone HZ}
Let $G=C_{2^k},k\geq 1$, then $H\Z_{neg}^G$ can be written as the direct sum of the following $3k-1+\frac{k(k-1)}{2}$ summands. (For simplicity, all indices $i,j$ have range $\geq 1$.)
\begin{enumerate}
    \item The torsion-free part ($k$ summands)
    \begin{equation*}
        \begin{aligned}
            &\ZZ\langle 2\uta^{-i}\rangle\\
            \oplus\,&\ZZ[\uta^{-1}]\langle4u_{\lambda_2}^{-i}\rangle\\
            \oplus\,&\ZZ[u_{\lambda_2}^{-1},\uta^{-1}]\langle8u_{\lambda_3}^{-i}\rangle\\
            \oplus\,&\cdots\\
            \oplus\,&\ZZ[u_{\lambda_{k-1}}^{-1},\cdots,\uta^{-1}]\langle 2^k u_{\lambda_k}^{-i}\rangle.
        \end{aligned}
    \end{equation*}

    \item The $2$-torsion part (involving $\aal$, $k$ summands)
    \begin{equation*}
        \begin{aligned}
            &\ZZ/2[\uta^{-1}][\aal^{-1},a_{\lambda_2}^{-1},\cdots,\alk^{-1}]\langle \Si\aal^{-1}\uta^{-1}\rangle\\
            \oplus\,&\ZZ/2[u_{\lambda_2}^{-1}][a_{\lambda_3}^{-1},\cdots,\alk^{-1}]\langle\Si{ a_{\lambda_2}^{-i}\uta^{-j}}\rangle\langle{\aal}\rangle\\
            \oplus\,&\ZZ/2[u_{\lambda_2}^{-1},u_{\lambda_3}^{-1}][a_{\lambda_4}^{-1},\cdots,\alk^{-1}]\langle\Si{ a_{\lambda_3}^{-i}\uta^{-j}}\rangle\langle{\aal}\rangle\\
            \oplus\,&\cdots\\
            \oplus\,&\ZZ/2[u_{\lambda_2}^{-1},\cdots,u_{\lambda_k}^{-1}]\langle\Si{ a_{\lambda_k}^{-i}\uta^{-j}}\rangle\langle{\aal}\rangle
        \end{aligned}
    \end{equation*}
    and (not involving $\aal$, $k-1$ summands)
    \begin{equation*}
        \begin{aligned}
            &\ZZ/2[a_{\lambda_3}^{-1},\cdots,\alk^{-1}]\langle\Si a_{\lambda_2}^{-i}\uta^{-j}\rangle\\
            \oplus\,&\ZZ/2[a_{\lambda_4}^{-1},\cdots,\alk^{-1}]\langle\Si a_{\lambda_3}^{-i}\uta^{-j}\rangle\\
            \oplus\,&\cdots\\
            \oplus\,&\ZZ/2\langle\Si \alk^{-i}\uta^{-j}\rangle.
        \end{aligned}
    \end{equation*}
    \item For a fixed $s$ where $2\leq s\leq k$, the $2^s$-torsion 
    part ($k-s+1$ summands)
    \begin{equation*}
        \begin{aligned}
            &\ZZ/{2^s}[\uta^{-1},\cdots,u_{\lambda_{s-1}}^{-1}][a_{\lambda_{s+1}}^{-1},\cdots,\alk^{-1}]\langle\Si a_{\lambda_s}^{-i}u_{\lambda_s}^{-j}\rangle\\
            \oplus\,&\ZZ/{2^s}[\uta^{-1},\cdots,u_{\lambda_{s-1}}^{-1}][a_{\lambda_{s+2}}^{-1},\cdots,\alk^{-1}]\langle\Si a_{\lambda_{s+1}}^{-i}u_{\lambda_s}^{-j}\rangle\\
            \oplus\,&\cdots\\
            \oplus\,&\ZZ/{2^s}[\uta^{-1},\cdots,u_{\lambda_{s-1}}^{-1}]\langle\Si a_{\lambda_k}^{-i}u_{\lambda_s}^{-j}\rangle.
        \end{aligned}
    \end{equation*}
\end{enumerate}
\end{theorem}
\begin{proof}
    The theorem is true for $k=1$ by \cite{Duggerthesis}, \cite{Gre18} and \cite{Zeng17}. It is true for $k=2$ by \cite{Yan22HZ}. Assume it is true for $k-1$. Note that since $\Z$ is a constant Mackey functor, for any $H\subset G$, we have that  $(H\Z_G)^{H}=H\Z_{G/H}$, where we use $H\Z_K$ to emphasise the $K$-equivariant $H\Z$. When $V\in RO(G)$ does not contain any copy of $\lambda_k$, by Corollary \ref{reduction to quotient group}, we get an isomorphism which we will denote by
    \[
    \varepsilon^*:H\Z_V^{C_{2^{k-1}}}\xrightarrow{\cong} H\Z^{G}_{V}
    \]
    Here for simplicity, we used $V$ for both the $C_{2^{k-1}}$ and the $C_{2^k}$-representation. Under this map, we have $\varepsilon^*(a_{\lambda_i})=a_{\lambda_{i}}$ and $\varepsilon^*(u_{\lambda_i})=u_{\lambda_{i}}$.
    
    Write $V=V^{\sharp}+a_k\lambda_k$ and note that $a_k\geq 0$ since we are in the negative cone. There are several different cases to consider.

    When $|V^{\sharp}|\geq 1$, then by Proposition \ref{prop alambda iso}, classes in $H\Z_{V^{\sharp}}^G$ are all infinitely $\alk$-divisible. For torsion classes in degrees with $|V^{\sharp}|=0$, by induction, we know they are all infinitely divisible by $a_{\lambda_{k-1}}$. Since $\res^G_e(a_{\lambda_{k-1}})=0$, these classes restrict to $0$ under $\res^G_e$. By Proposition \ref{prop action of alambda}, they are divisible by $\alk$. They are actually infinitely divisible by $\alk$, since for any $x\in H\Z_V^G$ with $|V^{\sharp}|=0$,  $\alk^{-1}x$ is in total degree $\geq 2$, thus infinitely $\alk$-divisible. As a result, we deduce that the torsion classes in $H\Z_V^G$ with $|V^{\sharp}|\geq 0$ are $\varepsilon^{*}H\Z_V^{C_{2^{k-1}}}[\alk^{-1}]$.

    For torsion-free classes in $H\Z_V^G$ with $|V^{\sharp}|\geq 0$, we know that $H\Z_{V^{\sharp}}^G$ consist of all the 'negative powers' of u's as in part $(1)$, with no $u_{\lambda_k}$ involved. By induction, we know that the Mackey functors they generate have $\ZZ$ in each level, with $\res^H_K$ equal to $1$ or $2$ when $[H:K]=2$. Thus restrictions are injective on these classes and they are not divisible by $\alk$.

    If $|V^{\sharp}|<0$, then let $c$ be the unique positive integer such that the total degree of $W=V^{\sharp}+c\lambda_k$ equals $0$ or $1$ depending on whether $|V^{\sharp}|$ is even or odd. Let $d_k=a_k-c$, so that we have $V=W+d_k\lambda_k$. 
    
    If $|W|=0$ and orientable, we get the torsion-free classes 
    \[\ZZ[u_{\lambda_{k-1}}^{-1},\cdots,\uta^{-1}]\langle 2^k u_{\lambda_k}^{-i}\rangle
    \]
    by Proposition \ref{prop u inverse}. They are not $\alk$ divisible as before. This completes the inductive proof of part $(1)$. If $|W|=0$ and non-orientable, we get the classes 
    \[
    \ZZ/2[u_{\lambda_2}^{-1},\cdots,u_{\lambda_k}^{-1}]\langle\Si{ a_{\lambda_k}^{-1}\uta^{-i}}\rangle\langle{\aal}\rangle
    \]
    by Proposition \ref{prop total dim 0 coho nonorient}. We will justify their names later. Since $H\Z_W^e=\Z$, the restrictions of these classes to the trivial subgroup are $0$, thus they are infinitely divisible by $\alk$. This finishes the inductive proof of the classes in part $(2)$ involving $\aal$.

    Now we consider the gradings with $|W|=1$. By 
    Proposition \ref{prop cohomology of total deg 1}, we have $H\Z_{U-|U|+\alpha}^G=0$ for any orientable $U$. Thus we only need to look at the groups $H\Z_{W}^G$ with $W=U-|U|+1$. From the cofiber sequence
    \[
    G/G'_+\xrightarrow{}S^0\xrightarrow{\aal}S^{\alpha},
    \]
    we can derive long exact sequences for some orientable $G$-representation $U$
    \begin{equation}
    \label{seq1}
    \tag{$\ast$}
    \cdots\to H\Z_{U+2-|U|}^{G'}\to H\Z_{U+\alpha+1-|U|}^G\xrightarrow{\aal}H\Z_{U-|U|+1}^G\xrightarrow{\res^G_{G'}}H\Z_{U-|U|+1}^{G'}\xrightarrow{\tr^G_{G'}}H\Z_{U-|U|+\alpha}^G\to\cdots,
    \end{equation}
and
    \begin{equation}
    \label{seq2}
    \tag{$\ast\ast$}
    \cdots\to H\Z_{U+1-|U|}^{G'}\to H\Z_{U+1-|U|}^G\xrightarrow{\aal}H\Z_{U-|U|+1-\alpha}^G\xrightarrow{\res^G_{G'}}H\Z_{U-|U|}^{G'}\xrightarrow{\tr^G_{G'}}H\Z_{U-|U|}^G\to\cdots.
    \end{equation}
We focus on the groups $H\Z_{U-|U|+1}^G$ and $H\Z_{U-|U|+1-\alpha}^G$ in the middle.

In the sequence \eqref{seq1}, the first group is $0$ by Theorem \ref{keythm}, since in the cellular cochain of $S^U$, $d^{|U-2|}=2^i$ for some $i$ and $\ZZ$ is torsion-free. The last group is $0$ by Proposition \ref{prop cohomology of total deg 1}. The group $H\Z_{U+\alpha+1-|U|}^G=\ZZ/2$ since in the cellular cochain of $S^{U+\alpha}$, $d^{|U|-1}=\overline{\zeta}_i$ for some $i$ which is $0$, and $d^{|U|-2}=1+\gamma=2$. Since each group is cyclic, as can be seen easily from the cellular cochain, we know the sequence \eqref{seq1} is always of the form
\begin{equation}
\label{ext of cyclic groups}
\begin{tikzcd}
    0\rar& \ZZ/2\rar["\aal"]& \ZZ/2^i\rar["\res^G_{G'}"]& \ZZ/2^{i-1}\rar& 0
    \end{tikzcd}
\end{equation}
for $1\leq i\leq k$. More concretely, from the cellular cochain, we have
\begin{equation*}
        H\M_{W}^G=\begin{cases}
             \ZZ/{2^k}&\text{if}\,c\geq 2,\\
             \ZZ/{2^i}&\text{if}\,c=1\,\text{and}\,\lambda_{max}(V^{\sharp})=\lambda_i,\\
             0&\text{if}\,c=1\,\text{and}\,V^{\sharp}=|V^{\sharp}|.
        \end{cases}
    \end{equation*}
Thus all the $2^{i-1}$-torsion groups in $H\Z^{G'}_{V}$ with $|W|\geq 1$ lift to $H\Z_{V}^G$ and become $2^i$-torsion, when $i\geq 2$. For example, we look at the $4$-torsion classes in $H\Z_{V}^G$ for $V$ in the negative cone. There are two sources of these classes. One is from the quotient group $C_{2^{k-1}}$ through $\varepsilon^*$ when we already have $|V^{\sharp}|\geq 0$. Another is from the subgroup $G'$ lifted via $\res^G_{G'}$, when $|V^{\sharp}|<0$ but $|W|=1$. By induction, we have the $4$-torsion classes in $H\Z_{neg}^{G'}$ ($k-2$ summands)
\begin{equation*}
        \begin{aligned}
            &\ZZ/4[\uta^{-1}][a_{\lambda_{3}}^{-1},\cdots,a_{\lambda_{k-1}}^{-1}]\langle\Si a_{\lambda_2}^{-i}u_{\lambda_2}^{-j}\rangle\\
            \oplus\,&\ZZ/4[\uta^{-1}][a_{\lambda_{4}}^{-1},\cdots,a_{\lambda_{k-1}}^{-1}]\langle\Si a_{\lambda_{3}}^{-i}u_{\lambda_2}^{-j}\rangle\\
            \oplus\,&\cdots\\
            \oplus\,&\ZZ/4[\uta^{-1}]\langle\Si a_{\lambda_{k-1}}^{-i}u_{\lambda_2}^{-j} \rangle
        \end{aligned}
    \end{equation*}
They give rise to classes in $H\Z_{neg}^G$ via $\varepsilon^*$, and we should join $[\alk^{-1}]$ since they are infinitely $\alk$-divisible. These are the first $k-2$ summands of the $4$-torsion classes in part $(3)$ for $C_{2^k}$. The last summand of the $4$-torsion classes are lifted via $\res^G_{G'}$ from the classes
\[\ZZ/{2}\langle\Si a_{\lambda_{k-1}}^{-i}u_{2\alpha}^{-j}\rangle
\]
in $H\Z_{neg}^{G'}$. These classes lifts to
\[\ZZ/{4}[\uta^{-1}]\langle\Si a_{\lambda_k}^{-i}u_{\lambda_2}^{-j}\rangle.
\]
Notice that we added $[u_{2\alpha}^{-1}]$ since $\res^G_{G'}(\uta)$ is a unit. The $2^s$-torsion classes for $2\leq s\leq k$ are proved the same way, and we finished the inductive proof of part $(3)$.

We are left with the proof of the $2$-torsion classes that does not involve $\aal$. The first $k-2$ summands comes from $\varepsilon^*$ followed by $\alk$-divisibility. The last summand
\[
\ZZ/2\langle\Si \alk^{-i}\uta^{-j}\rangle
\]
comes for the sequence \eqref{ext of cyclic groups} when $i=1$. From the cellular cochain, we know in the case of $|V^{\sharp}|< 0$ and $|W|=1$, $H\Z_W^G=\ZZ/2$ only if $W=*+a_1{\lambda_1}+\lambda_k,a_1\neq 0$. In these case, these groups are generated by
\[
\ZZ/2\langle\Si \alk^{-1}\uta^{-j}\rangle.
\]
Since they are $\alk$-divisible, we will finish the proof once we justify the notations.

As in \cite{Yan22HZ}, the classes
\[
\ZZ/2[\uta^{-1}]\langle \tr^{C_{2^k}}_{C_{2^{k-1}}}(e_{3\alpha})\alk^{-1}\rangle
\]
can be also written as
\[
\ZZ/2\langle \Si\aal^{-1}\alk^{-1}\uta^{-i}\rangle.
\]
The exact sequence \eqref{seq1} tells us
\[
\ZZ/2\langle \Si\aal^{-1}\alk^{-1}\uta^{-i}\rangle\cdot\aal= \ZZ/2\langle\Si \alk^{-1}\uta^{-i}\rangle
\]
since $H\Z_{U-|U|+1}^{G'}=0$. To justify the names of the classes
\[
    \ZZ/2[u_{\lambda_2}^{-1},\cdots,u_{\lambda_k}^{-1}]\langle\Si{ a_{\lambda_k}^{-1}\uta^{-i}}\rangle\langle{\aal}\rangle
\]
we look at the exact sequence \eqref{seq2}. We get 
\[
\ZZ/2\langle\Si \alk^{-1}\uta^{-i}\rangle\cdot\aal=\ZZ/2\langle\Si{ a_{\lambda_k}^{-1}\uta^{-i}}\rangle\langle{\aal}\rangle.
\]
This is because in the degrees $U-|U|$, transfers are either $1$ or $2$, and $\ker(\tr(U-|U|)^G_{G'})=0$.
\end{proof}

\begin{remark}
    The result has been verified additively by Anderson duality used in \cite{Zeng17}. We have $I_{\ZZ}(H\Z)\simeq \Sigma^{2-\lambda_k}H\Z$. Anderson duality gives the following short exact sequence of abelian groups
    \[
    0\to \Ext^1(H\Z_{\star-1}^G,\ZZ)\to H\Z_{\lambda_k-2-\star}^G\to \Hom_{\Ab}(H\Z_{\star}^G,\ZZ)\to 0.
    \]
    Torsion and torsion-free classes in $H\Z_{neg}^G$ will dualize to $H\Z_{pos}^G$ (sometimes outside the positive cone) via $\Ext^1(-,\ZZ)$ and $\Hom_{\Ab}(-,\ZZ)$ respectively. By an abuse of notation, we denote this duality by $D$. For example, we have the dualities
    \begin{equation*}
        \begin{aligned}
            &D(2^ku_{\lambda_k}^{-1})=1,\\
            &D(2^iu_{\lambda_i}^{-1})=2^{k-i}u_{\lambda_i}u_{\lambda_k}^{-1},\\
            &D(\Si a_{\lambda_i}^{-1}\uta^{-1})=\aal^2a_{\lambda_i}\alk^{-1},\\
            &D(\Si \aal a_{\lambda_i}^{-1}\uta^{-1})=\aal a_{\lambda_i}\alk^{-1},\\
            &D(\Si\alk^{-1}u_{\lambda_k}^{-1})=\alk,\\
            &D(\Si a_{\lambda_t}^{-1}u_{\lambda_s}^{-1})=a_{\lambda_t}a_{\lambda_s}\alk^{-1}(t\geq s\geq 2).
        \end{aligned}
    \end{equation*}
The divisibility relations between the Euler and orientation classes on the right hand sides of the equations can be seen by cellular computations. See for example \cite{HHRc}.
\end{remark}

\bibliographystyle{alpha}
\bibliography{bib}

\end{document}